\documentclass [leqno,11pt]{amsart}    
\usepackage{amssymb,amsmath,latexsym,amsfonts,amsbsy,amsthm,mathtools,graphicx,color,subeqnarray,cases,bm}      
\usepackage{float}     
\usepackage{hyperref}   
\usepackage{scalerel,stackengine}        
\stackMath     
\newcommand\reallywidehat[1]{%
\savestack{\tmpbox}{\stretchto{%
  \scaleto{%
    \scalerel*[\widthof{\ensuremath{#1}}]{\kern-.6pt\bigwedge\kern-.6pt}%
    {\rule[-\textheight/2]{1ex}{\textheight}}
  }{\textheight}%
}{0.5ex}}%
\stackon[1pt]{#1}{\tmpbox}%
}

\definecolor{myback}{RGB}{204,232,207}

\usepackage{cancel}
\usepackage[dvipsnames]{xcolor}
\numberwithin{equation}{section}

\allowdisplaybreaks

\numberwithin{equation}{section}

\allowdisplaybreaks

\let\na=\nabla

\let\pa=\partial

\def\no{\noindent}

\def\eqdef{\buildrel\hbox{\footnotesize def}\over =}

\newcommand{\beq}{\begin{equation}}
\newcommand{\eeq}{\end{equation}}
\newcommand{\ben}{\begin{eqnarray}}
\newcommand{\een}{\end{eqnarray}}
\newcommand{\beno}{\begin{eqnarray*}}
\newcommand{\eeno}{\end{eqnarray*}}

\newtheorem{theorem}{Theorem}[section]

\newtheorem{lemma}[theorem]{Lemma}
\newtheorem{proposition}[theorem]{Proposition}

\newtheorem{remark}[theorem]{Remark}

\begin{document}
\title{Viscosity driven instability of shear flows without boundaries}

\author{Hui Li}
\address[H. Li]{Department of Mathematics, New York University Abu Dhabi, Saadiyat Island, P.O. Box 129188, Abu Dhabi, United Arab Emirates.}
\email{lihuiahu@126.com, lihui@nyu.edu}

\author{Weiren Zhao}
\address[W. Zhao]{Department of Mathematics, New York University Abu Dhabi, Saadiyat Island, P.O. Box 129188, Abu Dhabi, United Arab Emirates.}
\email{zjzjzwr@126.com, wz19@nyu.edu}

\begin{abstract}
In this paper, we study the instability effect of viscous dissipation in a domain without boundaries. We construct a shear flow that is initially spectrally stable but evolves into a spectrally unstable state under the influence of viscous dissipation. To the best of our knowledge, this is the first result of viscosity driven instability that is not caused by boundaries. 
\end{abstract}

\maketitle

\section{Introduction}
Consider the two-dimensional incompressible Navier-Stokes equation on $\mathbb{T}\times\mathbb{R}$ (or $\mathbb{R}^2$)
\begin{equation}\label{eq:NS}
  \left\{
    \begin{array}{l}
    \pa_tV+V\cdot\na V+\na P-\nu\Delta V=0,\\
    \na\cdot V=0,\quad V(0,x,y)=V_{in}(x,y),
    \end{array}
  \right.
\end{equation}
where $\nu>0$ denotes the viscosity, $V=(V^{(1)},V^{(2)})$ is the velocity field and $P$ is the pressure. Let $W=\pa_xV_2-\pa_y V_1$ be the vorticity. The vorticity form of \eqref{eq:NS} is
\begin{align}\label{eq:NS1}
  \pa_tW+V\cdot\na W-\nu \Delta W=0.
\end{align}

When $\nu=0$, system \eqref{eq:NS} becomes the Euler equation for inviscid flow. Each shear flow $V_s=\left(b(y),0\right)$ is a steady solution of the Euler equation. 

However, for the Navier-Stokes equation, shear flows evolve according to the following heat equation:
\begin{align}\label{eq-heat}
  \pa_tb(t,y)-\nu\pa_{yy}b(t,y)=0,
\end{align}
and only the Couette flow $b(y)=y$ is a steady solution to \eqref{eq-heat}.

A natural question arises: if a perturbation is introduced around a shear flow, will the solution remain close to the original flow? This kind of questions gave rise to the  hydrodynamic stability theory, which can be traced back to the pioneering works of Rayleigh \cite{Rayleigh1880}, Kelvin \cite{Kelvin1887}, Reynolds \cite{reynolds1883}, Orr \cite{Orr1907}, and Sommerfeld \cite{Sommerfeld1908}. In this paper, we discuss the effect of viscosity on the stability of shear flows.


For the inviscid flow $\nu=0$, the linearized equation \eqref{eq:NS1} around $\left(b(y),0\right)$ is
\begin{align}\label{eq-linEuler}
  \pa_t\Delta \psi+b\pa_x\Delta\psi-b''\pa_x\psi=0,
\end{align}
where $\omega$ and $\psi=\Delta^{-1}\omega$ are the vorticity and stream function of the perturbation. Taking $\psi=\phi(y)e^{ik(x-ct)}$ with $k$ be the wave number in the $x$ direction and $c=c_r+ic_i$ the complex wave speed, we obtain from \eqref{eq-linEuler} the Rayleigh equation
\begin{align}\label{eq-Rayleigh}
  (b-c)\left(\pa_y^2-k^2\right)\phi_k-b''\phi_k=0,
 \end{align} 
with the boundary condition $\phi_k(y)\to0$ as $y\to\pm\infty$. So for shear flow, the stability problem reduced to study the spectrum of the Rayleigh operator
\begin{align*}
  \mathcal R_{b,k}=b \mathrm{Id}-b''(\pa_y^2-k^2)^{-1}.
\end{align*}
A shear flow is said to be spectrally unstable if it has an unstable eigenvalue, specifically when $kc_i>0$. A classical result of Rayleigh \cite{Rayleigh1880} provides a sufficient condition for the existence of such an unstable eigenvalue:  there must be an inflection point $y=y_s$ such that $b''(y_s)=0$. This condition was later improved by Fj\o{}rtoft \cite{fjortoft1950application}. Howard \cite{Howard1961} introduces the semicircle theorem, gives the possible distribution of the eigenvalue. Additionally, Squire's theorem \cite{squire1933} shows that for shear flows, spectral instability in 2D is equivalent to spectral instability in 3D.  

For the Couette flow without boundary, there is no point spectrum, and continuous spectrum distribute on $\mathbb R$. Orr \cite{Orr1907} observe that the velocity field decays back to equilibrium in $L^2$. This stability effect is called inviscid damping. The mechanism leading to the inviscid damping is the vorticity mixing driven by the shear flow.  Case \cite{Case1960} and Dikii \cite{dikii1964} discuss the linear inviscid damping for the channel cases. In the breakthrough paper \cite{BM2015},  Bedrossian-Masmoudi proved the nonlinear inviscid damping of the Couette flow when the perturbations are in the Gevrey-$m$ class ($1\leq m<2$). See also \cite{IonescuJia2020cmp} for the finite channel setting. 

For general monotonic shear flows, the continuous spectrum is also located on $\mathbb R$. Rosencrans and Sattinger \cite{RosSat1966} point out the stability of the continuous spectrum for monotonic flows. We say that a monotonic flow is spectrally stable if it has only continuous spectrum. However, the Rayleigh operator may not be self-adjoint,  the presence of a nonlocal term makes the stability problem more complex. The linear inviscid damping for such spectrally stable monotonic flows was rigorously proved by Wei-Zhang-Zhao \cite{WeiZhangZhao2018}. See also related results \cite{LinZeng2011,Zillinger2017jde,Jia2020siam,Jia2020arma}. Recent works by Ionescu-Jia \cite{IJ2020} and Masmoudi-Zhao \cite{MasmoudiZhao2020} have established the nonlinear inviscid damping for spectrally stable monotonic shear flows.  

For the viscous flow $\nu>0$, the vorticity mixing effect of shear flow leads to a stability mechanism known as enhanced dissipation, where the decay rate of vorticity is significantly higher than the diffusive decay rate. Enhanced dissipation was first observed by Kelvin \cite{Kelvin1887} in the study of Couette flow and has since been widely studied. An important related topic is known as the transition threshold. Since this is not the focus of the paper, we will not provide a detailed explanation. Interested readers can refer to the following references \cite{trefethen1993,BGM2017,BGM2020,BMV2016,LiWeiZhang2020,ChenWeiZhang2020,MasmoudiZhao2019,LiMasmoudiZhao2022critical,LMZ2022G,chen2019linear,CWZZ2023,bedrossian2024uniform} and the works cited therein.

There is a strong connection between the asymptotic stability of inviscid and viscous fluids. For the spectrally stable monotonic shear flows, Chen-Wei-Zhang \cite{CWZ2022} and Jia \cite{Jiahao2022} proved the linear asymptotic stability with viscosity in the finite channel $\mathbb{T}\times[0,1]$ and unbounded channel $\mathbb{T}\times\mathbb{R}$ respectively. It is important to note that any shear flow is a steady state of the Euler equations; however, for the Navier-Stokes equation, shear flows evolve according to \eqref{eq-heat}. In \cite{CWZ2022,Jiahao2022}, an external force is added to keep the shear flow unchanged. The authors of this paper proved the nonlinear asymptotic stability without  the need for an external force in \cite{li2023asymptotic}, under the assumption  that at each time the shear flow is spectrally stable:
\begin{itemize}
    \item[(\textasteriskcentered)] For any $t \geq 0$, $k\in\mathbb Z\setminus\{0\}$, the Rayleigh operator $\mathcal{R}_{b(t),k} = b(t,y)\mathrm{Id} -  b''(t,y)(\pa_y^2-k^2)^{-1}$ has no embedded eigenvalues or eigenvalues.
\end{itemize}
Very recently, \cite{beekie2024transition} presented the nonlinear stability without external force in a finite time, during which the shear flow remains very close to the spectrally stable initial state. See also the stability results for bounded monotonic shear flows in \cite{GNRS2020}.

Naturally, one may wonder whether it is possible for a shear flow to evolve from spectrally stable state to spectrally unstable state under \eqref{eq-heat}. Is the assumption (\textasteriskcentered) in \cite{li2023asymptotic} too strong, or could we prove asymptotic stability by placing the spectral assumption only on the initial shear flow? The external force used in \cite{CWZ2022,Jiahao2022} is very small, on the order of $\nu$, and the viscosity makes the non-zero modes decay faster than the inviscid case, resulting in spectrum gap for the channel scenario. According to the analysis in \cite{li2023asymptotic}, monotonic shear flows ultimately converge to a Couette flow which is highly stable. These indications suggest that viscosity may have a stabilizing effect on shear flows. However, in this paper, we will provide an example demonstrating that \eqref{eq-heat} can cause a spectrally stable shear flow to become spectrally unstable.

When studying the problem in a domain with boundaries, it is well known that viscosity may generate instability. For a shear flow that the corresponding Rayleigh operator has no point spectrum, the Orr-Sommerfeld operator may still exhibit unstable eigenvalues. The associated eigenfunction is known as the Tollmien-Schlichting (T-S) wave. The theoretical foundation of T-S wave is established by Heisenberg \cite{heisenberg1924stabilitat}, Tollmien \cite{tollmien1930entstehung}, Schlichting \cite{schlichting1933}, and  Lin \cite{LinCC1955}. The existence of T-S wave is rigorously proved by Grenier-Guo-Nguyen \cite{GGN2016,GGN2016D}, see also \cite{YangZhang2023,MWWZ2024} for the compressible case. Another instability mechanism in shear flows is known as wave breakdown \cite{landahl1975wave,itoh1981secondary}. In this process, a primary disturbance can create a quasi-paralle unstable flow inside the boundary layer, which subsequently leads to the rapid growth of a high-frequency wavepacket. For mathematical insights on the related results, see \cite{bian2024instability,bian2024boundary}. The instabilities mentioned above are caused by the combined effects of viscosity and boundaries. However, the instability discussed in this paper does not originate from boundaries.

\subsection{Main results}The main conclusions of this paper are as follows. The first result is for $x\in\mathbb T$, where the wave number $k\in\mathbb Z$.
\begin{theorem}\label{thm}
  Given small constant $\gamma>0$, for any $\nu>0$, there exists a shear flow $b(t,y)$ satisfying \eqref{eq-heat}, along with two time points $0<\tilde T<T$, such that:
\begin{itemize}
    \item the Rayleigh operator $\mathcal R_{b(t),k}$ associated with the shear flow $b(t,y)$ has no point spectrum for any wave number $k\in\mathbb Z\setminus\{0\}$ during $t\in[0,\tilde T)$;
    \item at $t=\tilde T=\tilde{C}(\gamma)\nu^{-1}$ with $k=\pm1$, the Rayleigh operator $\mathcal R_{b(\tilde T),\pm1}$ has an embedded eigenvalue;
    \item for $t\in(\tilde T,T]$ with $k=\pm1$ and $T=C(\gamma)\nu^{-1}$, the Rayleigh operator $\mathcal R_{b(t),\pm1}$ has a unique unstable spectrum $c=ic_i(t)$ such that $c_i(t)>0$ and $c_i(T)=\gamma$.
\end{itemize}
\end{theorem}

This result reflects a instability driven by viscosity that is not caused by boundaries, suggesting that the spectral assumption (\textasteriskcentered) for $t\ge0$ in \cite{li2023asymptotic} is necessary for proving asymptotic stability. The stability of shear flow is more complex in viscous fluids than in inviscid fluids, as viscosity can destabilize a stable shear flow or stabilize an unstable one.

\begin{remark}[Embedded eigenvalue]
  In this paper, we define a shear flow as spectrally stable if the Rayleigh operator has only continuous spectrum, and spectrally unstable if it has an eigenvalue \(c\) such that \(c_i > 0\). The shear flow \(b(t,y)\) in Theorem \ref{thm} transitions from spectrally stability to spectrally instability. However, in cases where an embedded eigenvalue exists with \(c_i = 0\), there arises another type of instability due to the non-self-adjoint nature of the Rayleigh operator, which leads to growth in polynomial form. We will discuss this kind of instability in future work.
\end{remark}

\begin{remark}[Viscous instability]
  The spectral stability discussed in Theorem \ref{thm} pertains to the Rayleigh operator, and the parameter $\gamma$ is independent of viscosity $\nu$ (although it can also be defined to depend on $\nu$). However, when viscosity $\nu$ is very small, it has been shown in \cite{li2023asymptotic,beekie2024transition} that nonlinear enhanced dissipation occurs during the time interval when $b(t,y)$ remains spectrally stable. And by adding external force, asymptotic stability for $b(0,y)$ can be obtained as in \cite{CWZ2022,Jiahao2022}. For the unstable background flow $b(T,y)$, \cite{Grenier2000} shows that the spectral instability in the Euler system implies nonlinear instability in the Navier-Stokes system. Therefore, the conclusion of Theorem \ref{thm} is sufficient to demonstrate that the viscous fluid system transitions from stable to unstable. In other words, the system becomes more sensitive to perturbations. 
\end{remark}

\begin{remark}[Secondary instability]
  The instability phenomenon discussed in this paper is somewhat similar to the secondary instability process in 3D Couette flow. For 3D Couette flow, the lift-up effect evolves a spectrally stable Couette flow into a spectrally unstable streak flow, which subsequently leads to instability. Similarly, the instability discussed here involves the evolution of a spectrally stable shear flow into a spectrally unstable shear flow under the effect of viscosity, which then triggers instability.
\end{remark}

\begin{remark}[Nonlinear problem]
   For the shear flow constructed in Theorem \ref{thm}, we plan to explore  its nonlinear instability in future work. Specifically, we aim to prove that disturbances in the non-zero mode $\omega_{\neq}(t,y)$ can have significant growth in the nonlinear system:
   \begin{align}\label{eq-non-unstable}
    \sup_{t\ge0}\left\|\omega_{\neq}(t)\right\|_{L^2}\gg\left\|\omega_{\neq}(0)\right\|_{L^2}.
   \end{align}
   We can construct a shear flow such that the time duration for which the unstable spectrum persists is significantly longer than the time it first appear, and with only one unstable eigenvalue. This makes the nonlinear instability \eqref{eq-non-unstable} highly likely to occur. However, since the system undergoes significant changes as it transitions from stability to instability, rigorously proving this type of nonlinear instability remains a considerable challenge.
\end{remark}
\begin{remark}[Nonlinear problem with arbitrary small force]
   If we add an external force, even if it is arbitrarily small, we can easily see the nonlinear instability for the shear flow in Theorem \ref{thm}. Specifically, there exists $T_1<T$ such that $c_i(T_1)=\frac{1}{2}\gamma$. Since $\gamma$ is independent of $\nu$, the unstable eigenfunction $\phi_{1,\frac{1}{2}\gamma}$ for $k=1$ has good regularity. For sufficiently small $\nu$, $b(t,y)$ remains very close to $b(T_1,y)$ on $t\in[T_1,T_1+1]$. Then by adding the  force $f(t,x,y)=\nu^N\chi_{t\in[T_1,T_1+1]}\Delta\left(\phi_{1,\frac{1}{2}\gamma}(y)e^{ix}\right)$ on the right side of the vorticity system \eqref{eq:NS1},  with $N$ as large as desired, we can inject a small perturbation primarily in the form $\frac{\left(e^{\frac{\gamma}{2}}-1\right)\nu^N}{\gamma/2}\Delta\left(\phi_{1,\frac{1}{2}\gamma}(y)e^{ix}\right)$ into \eqref{eq:NS1} at $t=T_1+1$. From the theory of \cite{Grenier2000}, this arbitrarily small perturbation will generate a growth of size $\tilde \delta=O(1)$ that depends only on $b(T_1,y)$.
\end{remark}

The second result is for $x\in\mathbb R$, where the wave number $k\in\mathbb R$.
\begin{theorem}\label{thm2}
  Given small constant $k_0>0$, for any $\nu>0$, there exists a shear flow $b(t,y)$ satisfying \eqref{eq-heat}, such that the Rayleigh operator $\mathcal R_{b(0),k}$ of the initial state $b(0,y)$ has no point spectrum for any wave number $k\in\mathbb R$, and at some time $T$,  the Rayleigh operator $\mathcal R_{b(T),k}$ has a unstable spectrum for $0<|k|<k_0$. 
\end{theorem}
\begin{remark}[Bounded monotonic shear flow]
  In the proofs of Theorems \ref{thm} and \ref{thm2}, we construct a uniformly strictly monotonic shear flow (see \eqref{eq-back} below) to demonstrate the destabilizing effect of viscosity. For this constructed shear flow, viscosity induces significant variations near the origin $|y|\le \frac{1}{C}$, and these locally variations affect the stability of the Rayleigh operator. Using the same idea, we can also construct a uniformly bounded shear flow $\left|b(t,y)\right|\le C$, for which similar conclusions to those in Theorems \ref{thm} and \ref{thm2} hold.
\end{remark}
\subsection{Main ideas}
We adopt a constructive proof. The shear flow we use is in the flowing form:
\begin{align}\label{eq-back}
  b_{M,\gamma_0,\gamma_1,\gamma_2}(t,y)=y+M \left(\int_{0}^y \frac{\gamma^2_0}{\sqrt{4\nu t+\gamma^2_0}}e^{-\frac{z^2}{4\nu t+\gamma^2_0}} dz -\gamma_2\int_{0}^y \frac{\gamma^2_0\gamma^3_1}{\sqrt{4\nu t+\gamma^2_0\gamma^2_1}}e^{-\frac{z^2}{4\nu t+\gamma^2_0\gamma^2_1}} dz\right),
\end{align}
which satisfies \eqref{eq-heat}. Here $\gamma_0,\gamma_1,\gamma_2>0$ are small constants, where $\gamma_1\ll\gamma_2$, and $M\approx 1$ will be determined in the proof, which can be chosen independent of the viscosity $\nu$.

The mechanism driving the instability in the background flow is as follows:

The main component of the shear flow, given by
\begin{align*}
  b_{main}(t,y)=y+M  \int_{0}^y \frac{\gamma^2_0}{\sqrt{4\nu t+\gamma^2_0}}e^{-\frac{z^2}{4\nu t+\gamma^2_0}} dz ,
\end{align*}
is initially spectrally unstable and evolves slowly over time by choosing suitable $M$ and $\gamma_0$. We find that adding a high-frequency part
\begin{align*}
  b_{high}(t,y)= -M\gamma_2\int_{0}^y \frac{\gamma^2_0\gamma^3_1}{\sqrt{4\nu t+\gamma^2_0\gamma^2_1}}e^{-\frac{z^2}{4\nu t+\gamma^2_0\gamma^2_1}} dz, 
\end{align*}
can make the overall shear flow $b_{M,\gamma_0,\gamma_1,\gamma_2}(t,y)$ spectrally stable at $t=0$. The high-frequency part $b_{high}(t,y)$ is particularly sensitive to the heat diffusion effect, leading it to dissipate significantly for $t\ge \nu^{-1}\gamma^2_0\gamma^2_1$. Consequently, this dissipation allows the shear flow to revert to the unstable state represented by $b_{main}(t,y)$. 

To prove Theorem \ref{thm}, we introduce the critical wave number $k_*(t)$ and the normalized neutral mode $\phi_{k_*}$. Specifically, at each time $t$, the Rayleigh operator associated with $b_{M,\gamma_0,\gamma_1,\gamma_2}(t,y)$ has an embedded eigenvalue ($c_i=0$) for a positive wave number $k_*(t)\in\mathbb R^+$, which need not be an integer. We refer to the corresponding normalized eigenfunction $\phi_{k_*}$ as the neutral mode. It holds that
\begin{align*}
  -k_*^2(t)=\min_{\substack{\left\|\phi\right\|_{L^2}=1,\\\phi\in H^1}}\left(\left\|\phi'\right\|_{L^2}^2+\int_{\mathbb R}\frac{b_{M,\gamma_0,\gamma_1,\gamma_2}''(t,y)}{b_{M,\gamma_0,\gamma_1,\gamma_2}(t,y)} \phi^2 dy \right),
\end{align*}
with equality holding precisely when $\phi=\phi_{k_*}$. 

We reduce the spectral stability problem of the shear flow $b_{M,\gamma_0,\gamma_1,\gamma_2}(t,y)$ to studying the evolution of the critical wave number $k_*(t)$ over time. By adjusting suitable parameters, we obtain a shear flow $b_{M,\gamma_0,\gamma_1,\gamma_2}(t,y)$ where the critical wave number satisfies $k_*(0)<1$, and the final critical wave number satisfies $k_*(T)>1$. According to the results in \cite{lin2003instability}, this implies that the Rayleigh operator for the initial state $b_{M,\gamma_0,\gamma_1,\gamma_2}(0,y)$ has only a continuous spectrum for $k\in\mathbb Z\setminus\{0\}$, while for the final state $b_{M,\gamma_0,\gamma_1,\gamma_2}(T,y)$, the Rayleigh operator has an unstable eigenvalue for $k=1$. 

To verify how the critical wave number evolves over time, we derive the precise profile of the neutral mode $\phi_{k_*}$, which is a even, positive, bell shaped function without zeros. As time goes on, the function $\left|\frac{b_{M,\gamma_0,\gamma_1,\gamma_2}''(t,y)}{b_{M,\gamma_0,\gamma_1,\gamma_2}(t,y)}\right|$ exhibits significant growth near the origin for $|y|\le \gamma_0\gamma_1$. This growth, through the initial neutral mode $\phi_{k_*(0)}$, leads to the increase in $k_*(t)$. We select an appropriate $M$ so that the critical wave number $k_*(t)$ evolves from an initial value $k_*(0)=1-\frac{1}{C}\gamma_1\gamma_2$ to a final value $k_*(T)=1+\frac{1}{C}\gamma_1\gamma_2$. 

The approach to proving Theorem \ref{thm2} is similar. In this case, at initial time, the shear flow is spectrally stable for all the $k\in\mathbb R$; hence, there is no critical wave number and neutral mode at beginning. The difficulty here lies in finding an appropriate function to replace $\phi_{k_*(0)}$ used in the proof of Theorem \ref{thm}, one that can reflects the effects of the evolution of $\left|\frac{b_{M,\gamma_0,\gamma_1,\gamma_2}''(t,y)}{b_{M,\gamma_0,\gamma_1,\gamma_2}(t,y)}\right|$. We construct a special function $\tilde \phi$, which is the neutral mode of a shear flow that is close to $b_{M,\gamma_0,\gamma_1,\gamma_2}(0,y)$. Then, through $\tilde \phi$, we see the emergence of a critical wave number $k_*(t)>0$ and track its evolution over time. The details of this part are provided in Section 2. 

In Section 3, we provide a quantitative estimate of the unstable eigenvalue for the shear flow $b_{M,\gamma_0,\gamma_1,\gamma_2}(T,y)$ in the case of the torus $(x\in\mathbb T)$.  Since the constructed shear flow is odd, the unstable eigenvalue is purely imaginary.  We analyze the Wronskian $W(c,k)$ of the Rayleigh equation, where its zeros correspond to the eigenvalues of the Rayleigh operator. We derive the dependence of the eigenvalue on the wave number, expressed as $c(k)=ic_i(k)$, which satisfies $c_i(k_*(T))=0$. We prove that $\pa_k c_i(k)\approx-\gamma_0$, leading to the conclusion that  $c_i(1)\approx \gamma_0\gamma_1\gamma_2$. This estimate of the eigenvalue lays the groundwork for proving nonlinear instability. 

\no{\bf Notations}: Through this paper, we will use the following notations. 

We use $C$ to denote a positive big enough constant that may be different from line to line, and use $f\lesssim g$ ($f\approx g$) to denote
\begin{align*}
  f\le C g\quad (\frac{1}{C}g\le f\le C g).
\end{align*}
We use $f\ll g$ to indicate $f\le \frac{1}{C}g$ for a sufficient big $C$.

 Given a function $f(t,y)$, we denote its spatial derivation with respect to $y$ as
\begin{align*}
  f'(t,y)=\pa_yf(t,y),\qquad f''(t,y)=\pa_y^2f(t,y).
\end{align*}

\section{Stability transition of the shear flow}
In this section, we show that by selecting appropriate $(M,\gamma_0,\gamma_1,\gamma_2)$, the odd shear flow $b_{M,\gamma_0,\gamma_1,\gamma_2}(t,y)$ given in \eqref{eq-back} evolves from  spectrally stable to spectrally unstable.

It follows from \eqref{eq-back} that
\begin{align*}
  b''_{M,\gamma_0,\gamma_1,\gamma_2}(t,y)=-2y\gamma^2_0 M\left(\frac{1}{(4\nu t+\gamma^2_0)^{\frac{3}{2}}}e^{-\frac{y^2}{4\nu t+\gamma^2_0}}-\gamma_2\frac{\gamma^3_1}{(4\nu t+\gamma^2_0\gamma^2_1)^{\frac{3}{2}}}e^{-\frac{y^2}{4\nu t+\gamma^2_0\gamma^2_1}}\right).
\end{align*}
As all $\gamma_0,\gamma_1$ are small constants, $\gamma_2<1$, it is evident that $b_{M,\gamma_0,\gamma_1,\gamma_2}(t,y)>0$ and $b''_{M,\gamma_0,\gamma_1,\gamma_2}(t,y)<0$ for $y>0$. Therefore, $b_{M,\gamma_0,\gamma_1,\gamma_2}(t,y)$ has only one inflection point $y=0$. Furthermore,
\begin{align*}
  -\frac{b''_{M,\gamma_0,\gamma_1,\gamma_2}(t,y)}{b_{M,\gamma_0,\gamma_1,\gamma_2}(t,y)}>0.
\end{align*}
Thus the shear flow we considering also belongs to $\mathcal K^+$ class studied in \cite{lin2003instability}.

 In the rest of this section, we simplify the notation of $b_{M,\gamma_0,\gamma_1,\gamma_2}(t,y)$ to $b_{M}(t,y)$, focusing on its dependence on on $M$ and $t$, unless otherwise stated.

We first fix $t$, and study the Rayleigh equation for each shear flow $b_{M}(t,y)$. In the Rayleigh equation \eqref{eq-Rayleigh}, the eigenvalue $c$ such that $c_i=0$ is known as the embedded eigenvalue. Since $b_{M}(t,y)$ has only one inflection point $y=0$, it follows from the Rayleigh technique that the embedded eigenvalue can only occur at the  inflection point, see \cite{WeiZhangZhao2018} and also Lemma 5.5 of \cite{LiMasmoudiZhao2022critical}. Thus the embedded eigenvalue could only be $c=0$. If $(\phi_k,k)$ is a solution to the Sturm-Liouville problem given by
\begin{align}
  &\phi''_k-\frac{b_{M}''}{b_{M}}\phi_k=k^2\phi_k,\label{eq-SL}\\
  & \phi_k(y)\to0 \text{ as } y\to\pm\infty,\label{eq-SL-bc}
\end{align}
then $(\phi_k,k)$ is a special solution (often referred to as a neutral mode) to the Rayleigh equation \eqref{eq-Rayleigh} with $c=0$. It is clear that if $(\phi_k,k)$ is a neutral mode, then $(\phi_k,-k)$ is also a neutral mode.  Without loss of generality, and unless otherwise specified,  we always assume $k>0$ through this paper.

By our setting, we can see that $\frac{b''_{M}(t,y)}{b_{M}(t,y)}$ is a compact perturbation of $\frac{d^2}{dy^2}$, so the operator
\begin{align}
  L_{M,t}=-\frac{d^2}{dy^2}+\frac{b''_{M}(t,y)}{b_{M}(t,y)}
\end{align}
has the same essential spectrum as $-\frac{d^2}{dy^2}$. The first negative eigenvalue of $L_{M,t}$ satisfies
\begin{align*}
  \lambda_{M,t}=\min_{\substack{\left\|\phi\right\|_{L^2}=1,\\\phi\in H^1}} \left\langle L_{M,t}\phi,\phi \right\rangle=\min_{\substack{\left\|\phi\right\|_{L^2}=1,\\\phi\in H^1}}\left(\left\|\phi'\right\|_{L^2}^2+\int_{\mathbb R}\frac{b''_{M}(t,y)}{b_{M}(t,y)} \phi^2 dy \right).
\end{align*}
For $\lambda_{M,t}=-k_*^2(M,t)<0$, we denote the corresponding eigenfunction by $\phi_{k_*(M,t)}$ which satisfies $\left\langle L_{M,t}\phi_{k_*(M,t)},\phi_{k_*(M,t)} \right\rangle=\lambda_{M,t}=-k_*^2(M,t)$ is a neutral mode. We call $k_*(M,t)$ the critical wave number.

\begin{remark}\label{rmk-unstable}
  For shear flow $b_{M}(t,y)$, if $\min\limits_{\substack{\left\|\phi\right\|_{L^2}=1,\\\phi\in H^1}} \left\langle L_{M,t}\phi,\phi \right\rangle=\lambda_{M,t}$, there exists neutral mode with $k_*(M,t)=\sqrt{-\lambda_{M,t}}$, then the corresponding Rayleigh operator $\mathcal R_{b_{M}(t,y),k}$ has a purely imaginary unstable eigenvalue for $0<|k|<k_*$. And $\mathcal R_{b_{M}(t,y),k}$ has no point spectrum for $|k|>\sqrt{-\lambda_{M,t}}$.
\end{remark}
The first part of this remark follows directly from Theorem 1.5 of \cite{lin2003instability}, and the second part is easy to check by using the technical in Lemma 3.4 and Remark 3.8 of \cite{lin2003instability}. 

\begin{lemma}\label{lem-nozero}
  For the shear flow $b_{M}(t,y)$, the operator $L_{M,t}$ can have at most one negative eigenvalue, and the corresponding eigenfunction has no zeros.
\end{lemma}
\begin{proof}
  First, we show that all the neutral mode $\phi_{k}$ solves \eqref{eq-SL} has no zeros by using the Rayleigh technique. As all the coefficients in \eqref{eq-SL} are real, we can assume that $\phi_{k}$ is real. If there is a $y_0\in(-\infty,0]$ such that $\phi_k(y_0)=0$, then we have
\begin{align*}
  -\int_{-\infty}^{y_0}|\phi_k|^2 dy-k^2 \int_{-\infty}^{y_0} \left|\phi_k\right|^2dy-\int_{-\infty}^{y_0} \frac{b''_{M}(t,y)}{b_{M}(t,y)}\left|\phi_k\right|^2  dy=0.
\end{align*}
Using integration by parts for the last term on the left side gives
\begin{align*}
  -\int_{-\infty}^{y_0} \left|\phi_k'-\frac{b'_{M}(t,y)}{b_{M}(t,y)} \phi_k\right|^2dy-k^2\int_{-\infty}^{y_0} \left|\phi_k\right|^2dy=0,
\end{align*}
which implies that $\phi_k=0$ on $(-\infty,y_0]$, and thus $\phi_k\equiv0$. If $L_{M,t}$ has at least two negative eigenvalue, then by the Sturm-Liouville theory (see, for example Theorem 10.12.1 in \cite{zettl2005sturm}),  the eigenfunction $\phi_k$ for second negative eigenvalue has one zero point, which contradicts with the fact that $\phi_k$ has no zero. 
\end{proof}

Now, we reduce the spectral stability problem of the shear flow $b_{M}(t,y)$ to studying the evolution of the critical wave number $k_*(M,t)$ over time. Our aim is to prove the following proposition, which is a key step for Theorem \ref{thm}.
\begin{proposition}\label{lem-neutral}
  Given small constants $0<\gamma_0,\gamma_1,\gamma_2$, with $\gamma_1\ll\gamma_2$, there exists suitable $M$ such that the critical wave number of shear flow $b_{M}(t,y)$ satisfies
  \begin{align*}
    k_*(M,0)=1-\frac{1}{C}\gamma_1\gamma_2,\ \pa_tk_*(M,t)>0 \text{ for }t\in[0,T],\ k_*(M,T)=1+\frac{1}{C}\gamma_1\gamma_2, 
  \end{align*}
  where $T=\nu^{-1}\gamma_0^2\gamma_1^2$. Acoordingly, there exists $\tilde T\in(0,T)$ such that $k_*(M,\tilde T)=1$.
\end{proposition}
\begin{remark}
  In the proof of Proposition \ref{lem-neutral}, we actually do not require $\gamma_2$ to be small; it is sufficient to take $\gamma_2<1$, such as $\gamma_2=\frac{1}{2}$. However, in the subsequent proof of Proposition \ref{lem-quanti}, due to certain technical reasons, we need $\gamma_2$ to be sufficiently small.
\end{remark}

The proof of Theorem \ref{thm2} is reduced to the following proposition.
\begin{proposition}\label{lem-neutral2}
  Given constants $0<\gamma_0,\gamma_1\ll1$, and $0<\gamma_2<\frac{1}{2}$, there exists suitable $M$ such that the shear flow $b_{M}(t,y)$ satisfies: at initial time, $\min_{\substack{\left\|\phi\right\|_{L^2}=1,\\\phi\in H^1}} \left\langle L_{M,0}\phi,\phi \right\rangle=0$, and the Rayleigh operator $\mathcal R_{b_{M}(0,y),k}$ has no neutral mode for any $k\in\mathbb R$; there exists $\tilde T$ such that $\mathcal R_{b_{M}(t,y),k}$ has positive critial wave number $k_*(M,t)>0$ with $\pa_tk_*(M,t)>0$ on $t\in(\tilde T,T]$, where $T=\nu^{-1}\gamma_0^2\gamma_1^2$; $k_*(M,T)=\frac{1}{C}\gamma_1\gamma_2$.
\end{proposition}

\subsection{Existence of neutral mode}
We first show that, for $0\le t\le T=\nu^{-1}\gamma_0^2\gamma_1^2$, we can adjust $M$ to determine the value of $\lambda_{M,t}$ (actually, this results still hold for longer time such that $t\le\nu^{-1}\gamma_0^2$). 

\begin{lemma}\label{lem-M}
  The negative eigenvalue $\lambda_{M,t}$ of $L_{M,t}$ depends continuously on $M$, and it holds that $\pa_M\lambda_{M,t}<0$ when $\lambda_{M,t}<0$. Given $0>\lambda\ge -2$, there exists $1\lesssim M\lesssim 1-\lambda$ such that $\lambda_{M,t}=\lambda$.
\end{lemma}
\begin{proof}
  We have the following weighted Poincare type inequality,
\begin{align*}
  -\int_{\mathbb R}\frac{b''_{M}(t,y)}{b_{M}(t,y)} \phi^2(y) dy=&\int_{\mathbb R}  \left|\int_{\mathbb R}\left|\frac{b''_{M}(t,y)}{b_{M}(t,y)}\right|^{\frac{1}{2}} \phi'(y')\chi_{y'\le y}dy'\right|^2 dy\\
  \le&\int_{\mathbb R}  \int_{\mathbb R}\left|\frac{b''_{M}(t,y)}{b_{M}(t,y)}\right|  \left|\phi'(y')\right|^2 dy'  dy\lesssim M\left\|\phi'\right\|_{L^2}^2.
\end{align*}
It follows that for $M\le \frac{1}{C}$, $\min_{\substack{\left\|\phi\right\|_{L^2}=1,\\\phi\in H^1}} \left\langle L_{M,t}\phi,\phi \right\rangle=0$. 

Let $\theta(y)=\left(\frac{\pi}{2}\right)^{-\frac{1}{4}}e^{-y^2}$. Recall that $\frac{b''_{M}(y)}{b_{M}(y)}<0$. By taking $\gamma_0$ small enough and $M\ll \frac{1}{\gamma_0}$ big enough, we have 
\begin{align*}
  \frac{b''_{M}(y)}{b_{M}(y)}\approx -\frac{M}{\gamma_0}\text{ for }|y|\le \gamma_0,
\end{align*}
and
\begin{align*}
  \left\|\theta^{'}\right\|_{L^2}^2+\int_{\mathbb R}\frac{b''_{M}(y)}{b_{M}(y)} \theta^{2} dy\le1-C^{-1}M,
\end{align*}
which implies that $\lambda_{M,t}\lesssim 1-C^{-1}M$. 

Also, we have
\begin{align*}
  &\pa_M\frac{b''_{M}(y)}{b_{M}(y)}
  =-2 \frac{y^2\gamma^2_0 M\left(\frac{1}{(4\nu t+\gamma^2_0)^{\frac{3}{2}}}e^{-\frac{y^2}{4\nu t+\gamma^2_0}}-\gamma_2\frac{\gamma^3_1}{(4\nu t+\gamma^2_0\gamma^2_1)^{\frac{3}{2}}}e^{-\frac{y^2}{4\nu t+\gamma^2_0\gamma^2_1}}\right)}{\left(b_{M}(y)\right)^2}\le0.
\end{align*}
Then, from the fact that all the neutral modes $\phi_k$ are real functions without zeros, we can conclude that $\lambda_{M,t}$ depends continuously on $M$ and $\pa_M\lambda_{M,t}<0$. Therefore, for any $0>\lambda\ge -2$ there exists $M$ such that $\lambda_{M,t}=\lambda$. 
\end{proof}


\subsection{Profile of the neutral mode} Next, we provide the exact profile of the neutral mode, which plays a critical role in analyzing the evolution of the critical wave number.
\begin{lemma}\label{lem-profile-phi}
  For critical wave number $0<k_*(M,t)\le2$, the normalized neutral mode $\phi_{k_*(M,t)}(y)$ has the following properties:
\begin{itemize}
    \item $\phi_{k_*(M,t)}(y)$ is an even, positive function with no zeros;
    \item $\phi_{k_*(M,t)}(y)$ decays monotonically on \( [0, +\infty) \);
    \item On $|y|\le \frac{1}{k_*(M,t)}$, $\frac{1}{C}|k_*(M,t)|^{\frac{1}{2}}\le\phi_{k_*(M,t)}(y)\le C |k_*(M,t)|^{\frac{1}{2}}$;
    \item For $|y|\ge \frac{1}{k_*(M,t)}$, $\phi_{k_*(M,t)}(y)\le  C|k_*(M,t)|^{\frac{1}{2}}e^{-\frac{1}{C}k_*(M,t)|y|}$.
\end{itemize}
where the constant $C$ does not depend on $M$ and $t$.
\end{lemma}
\begin{proof}
From Lemma \ref{lem-nozero}, we know that $\phi_{k_*(M,t)}(y)$ has no zeros, allowing us to take it as a positive function. Since $\frac{b_{M}''}{b_{M}}$ is even, we can have $\phi_{k_*(M,t)}(y)$ to be even as well. 

As $\phi_{k_*}$ is even and $\frac{M}{\gamma_0}\gg1$, we have $\phi_{k_*}'(0)=0$, $\phi_{k_*}''(0)=(\frac{b''_{M}(t,0)}{b_{M}(t,0)}+{k_*}^2)\phi_{k_*}(0)<0$, so $y=0$ is a maximum point of $\phi_{k_*}$. Let $y^*>0$ be the first zero point of  $\phi''_{k_*}(y)$ on \( [0, +\infty) \). It holds that
\begin{align*}
  b''_{M}(t,y^*) =-{k_*}^2b_{M}(t,y^*),
\end{align*}
which is 
\begin{align*}
  &-2y^*\gamma^2_0 M\left(\frac{1}{(4\nu t+\gamma^2_0)^{\frac{3}{2}}}e^{-\frac{y^{*2}}{4\nu t+\gamma^2_0}}-\gamma_2\frac{\gamma^3_1}{(4\nu t+\gamma^2_0\gamma^2_1)^{\frac{3}{2}}}e^{-\frac{y^{*2}}{4\nu t+\gamma^2_0\gamma^2_1}}\right)\\
  =&-{k_*}^2 \left(y^*+M \left(\int_{0}^{y^*} \frac{\gamma^2_0}{\sqrt{4\nu t+\gamma^2_0}}e^{-\frac{z^2}{4\nu t+\gamma^2_0}} dz -\gamma_2\int_{0}^{y^*} \frac{\gamma^2_0\gamma^3_1}{\sqrt{4\nu t+\gamma^2_0\gamma^2_1}}e^{-\frac{z^2}{4\nu t+\gamma^2_0\gamma^2_1}} dz\right)\right).
\end{align*}
As $t\le T$, we have $y^*\approx \gamma_0 \left|\ln({k_*}^2\gamma_0)\right|^{\frac{1}{2}}$. Then for $y\ge y^*$, the exponential terms $e^{-\frac{y^{2}}{4\nu t+\gamma^2_0}}$  and $e^{-\frac{y^{2}}{4\nu t+\gamma^2_0\gamma^2_1}}$ will be very small. One can easily check that 
\begin{align*}
  \pa_y \left(\frac{b''_{M}(t,y)}{b_{M}(t,y)}\right)=\frac{b'''_{M} b_{M}(t,y)-b''_{M} b_{M}'(t,y)}{b_{M}^2(t,y)}\approx \frac{My}{\gamma_0^3}e^{-\frac{y^{2}}{4\nu t+\gamma^2_0}} >0.
\end{align*}
Thus $y^*$ is the only positive zero point of $\phi''_{k_*}$, implying that $\phi''_{k_*}(y)>0$ for $y>y^*$. From the boundary condition \eqref{eq-SL-bc}, it follows that $\phi_{k_*}(y)$ decays monotonically on $[0,+\infty)$. 

Next, we give more precise profile of $\phi_{k_*}(y)$. As \eqref{eq-SL} is a second order ODE, it has two linearly independent solutions. Let 
\begin{align*}
  \phi^A_{k_*}(y)=\phi(k,y,c)=b_M(t,y)\phi_1(y,k,c_r), \text{ with }k=k_*, c=0,
\end{align*}
where $\phi(k,y,c)$ and $\phi_1(k,y,c_r)$ is given in Proposition \ref{prop-phi}. Then, $\phi^A_{k_*}$ is a solution of \eqref{eq-SL}, but it grows exponentially as $y\to\pm\infty$, thus failing to satisfy the boundary condition \eqref{eq-SL-bc}. The other solution has the expression
\begin{equation}\label{eq-phiB}
  \begin{aligned}   
  \phi^B_{k_*}(y)=&\frac{\phi^A_{k_*}(y)}{b_{M}'(t,0)}\int^{y}_{-\infty}\frac{b_{M}'(t,0)-b_{M}'(t,y')}{b_{M}(t,y')^2}dy'-\frac{\phi_1(y,k_*,0)}{b_{M}'(t,0)}\\
  &+\phi^A_{k_*}(y)\int^y_{-\infty}\frac{1}{b_{M}(t,y')^2}\left(\frac{1}{\phi_1^2(y',k_*,0)}-1\right)dy'.   
  \end{aligned}
\end{equation}
As stated in Lemma 5.6 of \cite{LiMasmoudiZhao2022critical}, if $k_*$ is a critical number, then $\phi^B_{k_*}(y)$ belongs to the energy space and decays exponentially as $y\to\pm\infty$. Therefore, it holds that
\begin{align}\label{eq-phik*}
  \phi_{k_*}(y)=-\frac{\phi^B_k(y)}{\left\|\phi^B_k\right\|_{L^2}}.
\end{align}
From Proposition \ref{prop-phi} and \eqref{eq-phiB}, one can easily check that 
\begin{align*}
  \phi^B_{k_*}(0)=-\frac{1}{b_{M}'(t,0)},\ \left|\phi^B_{k_*}(y)\right|\approx 1 \text{ for }y=\pm \frac{1}{k_*},\ \left|\phi^B_{k_*}(y)\right|\le C e^{-\frac{1}{C}|{k_*}||y|}\text{ for } |y|\ge \frac{1}{k_*}.
\end{align*}
As $\left|\phi^B_{k_*}(y)\right|$ decays monotonically on $[0, +\infty)$, we have $\left\|\phi^B_{k_*}\right\|_{L^2}\approx k_*^{-\frac{1}{2}}$. The results of this lemma follows from \eqref{eq-phik*}.
\end{proof}

\subsection{Proof of Proposition \ref{lem-neutral}}
Now we show the evolution of critical wave number for $k_*(M,t)\approx 1$ and give the proof of Proposition \ref{lem-neutral}.

From Lemma \ref{lem-M}, for any $\frac{1}{2}\le\tilde k\le 2$, we can choose suitable  $M\approx1$ such that $k_*(M,0)=\tilde k$. We will show that for $0<t\le T$, it holds that
\begin{align}\label{eq-est-0-T}
  \int_{\mathbb R} \left(\frac{b''_{M}(t,y)}{b_{M}(t,y)}-\frac{b''_{M}(0,y)}{b_{M}(0,y)}\right) \phi_{k_*(M,0)}^2 dy\le-C_1\frac{\gamma_2\nu t}{\gamma_0^2\gamma_1},
\end{align}
where $C_1$ is a uniform constant for $\frac{1}{2}\le k_*(M,0)\le 2$.

We write
\begin{align*}
  &\frac{b''_{M}(t,y)}{b_{M}(t,y)}-\frac{b''_{M}(0,y)}{b_{M}(0,y)}\\
  =&\frac{\left(b''_{M}(t,y)-b''_{M}(0,y)\right)b_{M}(0,y) }{b_{M}(t,y)b_{M}(0,y)}-\frac{b''_{M}(0,y)\left(b_{M}(t,y)-b_{M}(0,y)\right)}{b_{M}(t,y)b_{M}(0,y)}\\
  =&\frac{2y\gamma^2_0 M\left(\gamma_2\frac{\gamma^3_1}{(4\nu t+\gamma^2_0\gamma^2_1)^{\frac{3}{2}}}e^{-\frac{y^2}{4\nu t+\gamma^2_0\gamma^2_1}}-\gamma_2\frac{1}{\gamma^3_0}e^{-\frac{y^2}{\gamma^2_0\gamma^2_1}}\right)}{b_{M}(t,y)}-\frac{2y\gamma^2_0 M\left(\frac{1}{(4\nu t+\gamma^2_0)^{\frac{3}{2}}}e^{-\frac{y^2}{4\nu t+\gamma^2_0}}-\frac{1}{\gamma^3_0}e^{-\frac{y^2}{\gamma^2_0}} \right)}{b_{M}(t,y)}\\
  &-\frac{b''_{M}(0,y)\left(b_{M}(t,y)-b_{M}(0,y)\right)}{b_{M}(t,y)b_{M}(0,y)}\\
  \eqdef&B_1+B_2+B_3.
\end{align*}

The main contribution is from $B_1$. Let
\begin{align*}
  h_1(t,y)=\gamma_2\frac{\gamma^3_1}{(4\nu t+\gamma^2_0\gamma^2_1)^{\frac{3}{2}}}e^{-\frac{y^2}{4\nu t+\gamma^2_0\gamma^2_1}}-\gamma_2\frac{1}{\gamma^3_0}e^{-\frac{y^2}{\gamma^2_0\gamma^2_1}},
\end{align*}
we can see that $h_1(t,y)$ is even in $y$, $h_1(t,0)<0$ for $t>0$, and its integration is negative
\begin{align*}
  \int_{\mathbb R} h_1(t,y) dy= \sqrt\pi \gamma_2\left(\frac{\gamma_1^3}{4\nu t+\gamma_0^2\gamma_1^2}-\frac{\gamma_1}{\gamma_0^2}\right)=-\sqrt\pi \gamma_2\frac{4\nu t\gamma_1}{(4\nu t+\gamma_0^2\gamma_1^2)\gamma_0^2}.
\end{align*}
It is clear that $h_1(t,y)$ has only one zero point $\tilde y$ on $[0, +\infty)$, which satisfies
\begin{align*}
  h_1(t,\tilde y)=0 \text{ for }\tilde y=\sqrt{\frac{(4\nu t+\gamma_0^2\gamma_1^2)\gamma_0^2\gamma_1^2}{4\nu t}\ln{\frac{(4\nu t+\gamma_0^2\gamma_1^2)^{\frac{3}{2}}}{\gamma_0^3\gamma_1^3}}}
\end{align*}
For $0<t\le T$, it is easy to check that $\sqrt{\frac{3}{2}}\gamma_0\gamma_1\le\tilde y\lesssim\gamma_0\gamma_1$. Next, we calculate the integration of negative part of $h_1(t,y)$. We have
\begin{align*}
  \int_{0}^{\tilde y}h_1(t,y) dy=&\gamma_2 \frac{\gamma^3_1}{(4\nu t+\gamma^2_0\gamma^2_1)^{\frac{3}{2}}}\int_{0}^{\tilde y}e^{-\frac{y^2}{4\nu t+\gamma^2_0\gamma^2_1}} dy-\gamma_2\frac{1}{\gamma^3_0}\int_{0}^{\tilde y}e^{-\frac{y^2}{\gamma^2_0\gamma^2_1}} dy\\
  =&\gamma_2\frac{\gamma^3_1}{(4\nu t+\gamma^2_0\gamma^2_1)}\int_{0}^{\frac{\tilde y}{(4\nu t+\gamma^2_0\gamma^2_1)^{\frac{1}{2}}}}e^{-z^2} dz-\gamma_2\frac{\gamma_1}{\gamma^2_0}\int_{0}^{\frac{\tilde y}{\gamma_0\gamma_1}}e^{-z^2} dz\\
  =&-\gamma_2\left(\frac{\gamma_1}{\gamma^2_0}-\frac{\gamma^3_1}{(4\nu t+\gamma^2_0\gamma^2_1)}\right)\int_{0}^{\frac{\tilde y}{\gamma_0\gamma_1}}e^{-z^2} dz-\gamma_2\frac{\gamma^3_1}{(4\nu t+\gamma^2_0\gamma^2_1)}\int_{\frac{\tilde y}{(4\nu t+\gamma^2_0\gamma^2_1)^{\frac{1}{2}}}}^{\frac{\tilde y}{\gamma_0\gamma_1}}e^{-z^2} dz\\
  \ge&- \frac{4\gamma_2\gamma_1\nu t}{\gamma^2_0(4\nu t+\gamma^2_0\gamma^2_1)}\frac{\tilde y}{\gamma_0\gamma_1}-\frac{\gamma_2\gamma^3_1}{(4\nu t+\gamma^2_0\gamma^2_1)}\left(\frac{\tilde y}{\gamma_0\gamma_1}-\frac{\tilde y}{(4\nu t+\gamma^2_0\gamma^2_1)^{\frac{1}{2}}}\right)\\
  \gtrsim& -\frac{\gamma_2\nu t}{\gamma^4_0\gamma_1}.
\end{align*}
Thus, we have
\begin{align*}
  \left|\int_{-\tilde y}^{\tilde y} h_1(t,y) dy\right| \le C\left|\int_{\mathbb R} h_1(t,y)dy\right|,
\end{align*}
which means that the integration of the negative part is in the same size of the whole integration. It follows that
\begin{align}\label{eq-neg-part}
  \left|\int_{-\tilde y}^{\tilde y} h_1(t,y) dy\right|\ge (1+\frac{1}{C})\left|\int_{\mathbb R\setminus[-\tilde y,\tilde y]} h_1(t,y) dy\right|
\end{align}

From Lemma \ref{lem-profile-phi}, we know that $\phi_{k_*(M,0)(y)}\approx 1$ for $|y|\le \frac{1}{2}$ has uniform lower bound with $\frac{1}{2}\le k_*(M,0)\le 2$, and decays monotonically on $[0,+\infty)$. Recall that $\frac{2\gamma^2_0 M}{1+M\gamma_0}\le\frac{2y\gamma^2_0 M}{b_{M}(t,y)}\le 2\gamma^2_0 M$, and $\tilde y\approx\gamma_0\gamma_1\ll1$. We deduce from \eqref{eq-neg-part} that
\begin{align*}
  \int_{\mathbb R}B_1 \phi_{k_*(M,0)}^2dy\lesssim \phi_{k_*(M,0)}(\tilde y) \gamma^2_0 M\left(\frac{\int_{-\tilde y}^{\tilde y}h_1(t,y)dy}{1+M\gamma_0}+\int_{\mathbb R\setminus[-\tilde y,\tilde y]}h_1(t,y)dy\right)\le-\frac{1}{C}\frac{\gamma_2\nu t}{\gamma^2_0\gamma_1}.
\end{align*}

The estimate for $B_2$ is very similar to $B_1$. Let
\begin{align*}
  h_2(t,y)=\frac{1}{(4\nu t+\gamma^2_0)^{\frac{3}{2}}}e^{-\frac{y^2}{4\nu t+\gamma^2_0}}-\frac{1}{\gamma^3_0}e^{-\frac{y^2}{\gamma^2_0}}.
\end{align*}
It holds that
\begin{align*}
   \int_{\mathbb R} \min(h_2(t,y),0) dy \approx \int_{\mathbb R} h_2(t,y)dy \approx-\frac{\nu t}{\gamma^4_0}.
\end{align*}
Then by the same argument, we have
\begin{align*}
  \left|\int_{\mathbb R}B_2 \phi_{k_*(M,0)}^2dy\right|\lesssim \frac{ \nu t}{\gamma^2_0 }.
\end{align*}

The last term $B_3$ is also an error term. Recall that $b_M(t,y)$ satisfis the heat equation \eqref{eq-heat}, we have
\begin{align*}
  \left|B_3\right|=&\left|\frac{b''_{M}(0,y)\left(b_{M}(t,y)-b_{M}(0,y)\right)}{b_{M}(t,y)b_{M}(0,y)}\right|=\left|\frac{b''_{M}(0,y)\int^t_0\pa_tb_{M}(t',y) dt'}{b_{M}(t,y)b_{M}(0,y)}\right|\\
  =&\left|\frac{b''_{M}(0,y)\nu\int^t_0b_{M}''(t',y) dt'}{b_{M}(t,y)b_{M}(0,y)}\right|\lesssim \frac{\nu t}{\gamma_0^2}e^{-\frac{y^2}{\gamma_0^2}}.
\end{align*}

Combining the above estimates and using the fact that $\gamma_1\ll\gamma_2$, we obtain \eqref{eq-est-0-T}.

From the proof, we can see that $C_1$ in is an uniform constant for $\frac{1}{2}\le k_*(M,0)\le 2$. We can take suitable $M$ such that $\lambda_{M,0}=-k_*^2(M,0)=-1+\frac{C_1}{2}\gamma_1\gamma_2$, and
\begin{align*}
  \lambda_{M,T}=&-k_*^2(M,T)=\min_{\substack{\left\|\phi\right\|_{L^2}=1,\\\phi\in H^1}}\left(\left\|\phi'\right\|_{L^2}^2+\int_{\mathbb R}\frac{b''_{M}(T,y)}{b_{M}(T,y)} \phi^2 dy \right)\\
  \le&\left\|\phi_{k_*(M,0)}'\right\|_{L^2}^2+\int_{\mathbb R}\frac{b''_{M}(T,y)}{b_{M}(T,y)} \phi_{k_*(M,0)}^2 dy\le \lambda_{M,0}- C_1\gamma_1\gamma_2=-1-\frac{C_1}{2}\gamma_1\gamma_2.
\end{align*}

If we replace $\phi_{k_*(M,0)}$ in \eqref{eq-est-0-T} with the neutral mode $\phi_{k_*(M,T)}$ corresponding to time $T$ and $k_*(M,T)$, we can similarly prove that
\begin{align*}
  \lambda_{M,0}\le \lambda_{M,T}+\frac{1}{C}\gamma_1\gamma_2.
\end{align*}

Moreover, from the proof of \eqref{eq-est-0-T}, we also have
\begin{align}\label{eq-est-t-T}
  \int_{\mathbb R} \left(\frac{b''_{M}(t_2,y)}{b_{M}(t_2,y)}-\frac{b''_{M}(t_1,y)}{b_{M}(t_1,y)}\right) \phi_{k_*(M,t_1)}^2 dy\le-\frac{1}{C}\frac{\gamma_2\nu (t_2-t_1)}{\gamma_0^2\gamma_1},
\end{align}
it follows that $\pa_tk_*(M,t)>0 \text{ for }t\in[0,T]$.

This finishes the proof of Proposition \ref{lem-neutral}. 

\subsection{Proof of Proposition \ref{lem-neutral2}}
Recalling the proof of Lemma \ref{lem-profile-phi}, there are two linearly independent solutions to \eqref{eq-SL} when $k=0$,
\begin{align*}
  \phi^A_0(y)&=b_{M}(t,y),\\
  \phi^B_0(y)&=\frac{b_{M}(t,y)}{b_{M}'(t,0)}\int^{y}_{-\infty}\frac{b_{M}'(t,0)-b_{M}'(t,y')}{b_{M}(t,y')^2}dy'-\frac{1}{b_{M}'(t,0)}.
\end{align*}
It is clear that $\phi^B_0(y)\approx-\frac{1}{b_{M}'(t,0)}$ as $y\to-\infty$ and $\phi^B_0(y)\approx y$ for $y\to+\infty$. Therefore, neither $\phi^A_0(y)$ nor $\phi^B_0(y)$, nor any linear combinations of them, belongs to the energy space. Thus for any $M$ and $t$ there is no neutral mode for $k=0$. 

Then, from Lemma \ref{lem-M}, we see that there is a critical value $M_0$ such that 
\begin{align*}
  \min_{\substack{\left\|\phi\right\|_{L^2}=1,\\\phi\in H^1}} \left\langle L_{M,0}\phi,\phi \right\rangle=0,\text{ for }M\le M_0,\quad
  \min_{\substack{\left\|\phi\right\|_{L^2}=1,\\\phi\in H^1}} \left\langle L_{M,0}\phi,\phi \right\rangle<0,\text{ for }M> M_0.
\end{align*}
Our goal is to show that $\min_{\substack{\left\|\phi\right\|_{L^2}=1,\\\phi\in H^1}} \left\langle L_{M_0,T}\phi,\phi \right\rangle\le-k_0^2$ for some $k_0>0$. The proof of Proposition \ref{lem-neutral} relies on the initial neutral mode $\phi_{k_*(M,0)}$, but in this case, no initial nuetral mode exists. The main challenge is to identify a suitable function showing that $L_{M,T}$ has a negative eigenvalue.

For extremely small $\delta_0$, $L_{M_0+\delta_0,0}$ has a negative eigenvalue $\lambda_{M_0+\delta_0,0}<0$, which is also extremely small. The exact value of $\delta_0$ and $\lambda_{M_0+\delta_0,0}$ are not crucial, we can treat them as $0+$. We simply need an starting point where the neutral mode exists. Next, for $M=M_0+\delta$ with $\delta\ge\delta_0$, we have $\lambda_{M_0+\delta,0}<\lambda_{M_0+\delta_0,0}$. We now show the dependent of $\lambda_{M,0}$ on $M$.

For $M\ge M_0+\delta_0$ and $\varepsilon>0$, it is clear that 
\begin{align*}
  -\frac{\lambda_{M+\varepsilon,0}-\lambda_{M,0}}{\varepsilon}\ge-\int_{\mathbb R} \frac{\left(\frac{b''_{M+\varepsilon }(0,y)}{b_{M+\varepsilon }(0,y)}-\frac{b''_{M}(0,y)}{b_{M}(0,y)}\right)}{\varepsilon}\phi_{k_*(M,0)}^2 dy.
\end{align*}
Taking $\varepsilon\to0+$, we deduce from Lemma \ref{lem-profile-phi} that
\begin{align*}
  -\pa_M\lambda_{M,0}\ge&-\int_{\mathbb R} \left(\pa_M \frac{b''_{M}(0,y)}{b_{M}(0,y)}\right)\phi_{k_*(M,0)}^2 dy\\
  =&\int_{\mathbb R}2 \frac{y^2\gamma^2_0 M\left(\frac{1}{(4\nu t+\gamma^2_0)^{\frac{3}{2}}}e^{-\frac{y^2}{4\nu t+\gamma^2_0}}-\gamma_2\frac{\gamma^3_1}{(4\nu t+\gamma^2_0\gamma^2_1)^{\frac{3}{2}}}e^{-\frac{y^2}{4\nu t+\gamma^2_0\gamma^2_1}}\right)}{\left(b_{M}(y)\right)^2}\phi_{k_*(M,0)}^2dy\\
  \gtrsim&k_*(M,0)=\sqrt{-\lambda_{M,0}},
\end{align*}
which implies $\pa_Mk_*(M,0)\gtrsim 1$. Therefore, for some small $\delta_1$ which will be determined later, we have for $M=M_0+\delta_1$, the critical wave number $k_*(M_0+\delta_1,0)\gtrsim \delta_1$. The corresponding neutral mode is $\phi_{k_*(M_0+\delta_1,0)}$, which is the desired funciton.

From the assumption on $M_0$, it holds that
\begin{align*}
  \left\langle L_{M_0,0}\phi_{k_*(M_0+\delta_1,0)},\phi_{k_*(M_0+\delta_1,0)} \right\rangle>0.
\end{align*}
Then, using Lemma \ref{lem-profile-phi} again, we have
\begin{align*}
  k_*^2(M_0+\delta_1,0)<&\left\langle L_{M_0,0}\phi_{k_*(M_0+\delta_1,0)},\phi_{k_*(M_0+\delta_1,0)} \right\rangle+k_*^2(M_0+\delta_1,0)\\
  =&\int_{\mathbb R} \left(\frac{b''_{M}(0,y)}{b_{M}(0,y)}-\frac{b''_{M_0+\delta_1}(0,y)}{b_{M_0+\delta_1}(0,y)}\right) \phi_{k_*(M,0)}^2dy\\
  \approx&\int_{\mathbb R} \frac{\delta_1}{\gamma_0}e^{-\frac{y^2}{\gamma_0^2}} k_*(M_0+\delta_1,0)dy\approx \delta_1k_*(M_0+\delta_1,0).
\end{align*}
Recall the fact that $k_*(M_0+\delta_1,0)\gtrsim \delta_1$. It holds that $k_*(M_0+\delta_1,0)\approx\delta_1$. Accordingly,
\begin{align*}
  \left\langle L_{M_0,\gamma_0,\gamma_1,0}\phi_{k_*(M_0+\delta_1,0)},\phi_{k_*(M_0+\delta_1,0)} \right\rangle\gtrsim -\delta_1^2.
\end{align*}

By applying the same argument in Proposition \ref{lem-neutral}, we obtain
\begin{align*}
  \int_{\mathbb R} \left(\frac{b''_{M_0}(T,y)}{b_{M_0}(T,y)}-\frac{b''_{M_0}(0,y)}{b_{M_0}(0,y)}\right) \phi_{k_*(M_0+\delta_1,0)}^2dy\approx-\frac{\gamma_2\nu T}{\gamma_0^2\gamma_1}\delta_1 \approx-\gamma_1\gamma_2 \delta_1.
\end{align*}
The only difference is that, in Proposition \ref{lem-neutral}, we have $k_*\approx1$ and $\phi_{k_1}\approx1$ for $|y|\le1$; here we have $k_*\approx\delta_1$ and $\phi_{k_1}\approx \delta_1^{\frac{1}{2}}$ for $|y|\le1$.

Therefore, for $\delta_1\ll\gamma_1\gamma_2$, we prove that
\begin{align}\label{eq-kT}
   \left\langle L_{M_0,T}\phi_{k_*(M_0+\delta_1,0)},\phi_{k_*(M_0+\delta_1,0)} \right\rangle\approx -\gamma_1\gamma_2 \delta_1.
\end{align}
This shows the existance of a positive critical wave number. This result holds for any $\delta_1>0$, therefore for any $t>0$ there exists $k_*(M_0,t)>0$.

We can take $\delta_1=\frac{1}{C}\gamma_1\gamma_2$ where $C$ is sufficiently large to ensure that \eqref{eq-kT} still holds. Consequently, the critical wave number can reach $k_0=\frac{1}{C}\gamma_1\gamma_2$. For $b_{M_0-\varepsilon}(t,y)$ with $0<\varepsilon\ll \gamma_1\gamma_2$, a positive critical wave number will also emerge as time evolves, with the sharp transition time being some $\tilde T>0$.  

Using similar techniques as in \eqref{eq-est-t-T}, we also have $\pa_tk_*(M,t)>0$ for any $t\in(0,T]$ such that $k_*(M,t)>0$. This finished the proof of Proposition \ref{lem-neutral2}.

\section{Quantitative estimate for the unstable eigenvalue} 
In this section, we analyze the unstable shear flow $b_{M,\gamma_0,\gamma_1,\gamma_2}(T,y)$  obtained in Proposition \ref{lem-neutral} and provide a quantitative estimate of the unstable eigenvalue $c$ for the Rayleigh operator $\mathcal R_{b_{M,\gamma_0,\gamma_1,\gamma_2}(T,y),k}$ on $k=\pm1$ modes. 
\begin{proposition}\label{lem-quanti}
  For the shear flow $b_{M,\gamma_0,\gamma_1,\gamma_2}(T,y)$ in Proposition \ref{lem-neutral}, the Rayleigh operator $\mathcal R_{b_{M,\gamma_0,\gamma_1,\gamma_2}(T,y),\pm1}$ has a unique unstable eigenvalue $c=ic_i=i\frac{1}{C}\gamma_0\gamma_1\gamma_2$.
\end{proposition}

In what follows, we simplify $b_{M,\gamma_0,\gamma_1,\gamma_2}(T,y)$ to $b(y)$. 

The proof of Proposition \ref{lem-quanti} consists of the following steps. First, we show that $c$ is an eigenvalue of $\mathcal R_{b,k}$ if and only if the Wronskian $W(c,k)=0$, as defined in \eqref{eq-Wron}, and we provide some key properties of $W(c,k)$. Then, by applying the implicit function theorem, we establish the existence of a implicit curve $c_i(k)$ with $c_i(k_*(T))=0$ that
\begin{align}\label{eq-curve}
  W_r(i c_i(k),k)=0,\quad\frac{d c_i(k)}{dk}\le -\frac{1}{C}\gamma_0.
\end{align}
Since $k_*(T)=1+\frac{1}{C}\gamma_1\gamma_2$, this implies that $\mathcal R_{b,1}$ has an unstable eigenvalue $c=i\frac{1}{C}\gamma_0\gamma_1\gamma_2$. In the last step, we show the uniqueness of the unstable eigenvalue. 

In the remainder of this section, we provide the detailed proof of the steps outlined above. The proof is based on the special solution $\phi(y,k,c)$, given in Proposition \ref{prop-phi}, which solves \eqref{eq-Rayleigh}. Here we adopt a different notation from Section 2 to emphasize the solution's dependence on  $k$ and $c$. 

Since the solution $\phi(y,k,c)$ grows exponentially as $y\to\pm\infty$, we can define
\begin{align}\label{eq-psipm}
  \varphi^\pm(y,k,c)=&\phi(y,k,c)\int^y_{\pm\infty}\frac{1}{\phi^2(y',k,c)}dy'
\end{align}
which are independent solutions of \eqref{eq-Rayleigh} corresponding to $\phi(y,k,c)$, and $\varphi^\pm(y,k,c)$ decay to zero as $y\to\pm\infty$ separately.

Next, we introduce the Wronskian $W(c,k)=W_r(c,k)+iW_i(c,k)$. 
\begin{align}\label{eq-Wron}
  W(c,k)=\int_{-\infty}^{+\infty} \frac{1}{\phi^2(y',k,c)} dy'=
  \det\left(\begin{matrix}\varphi^{-}(y,k,c)&\varphi^{+}(y,k,c)\\
 \pa_y\varphi^{-}(y,k,c)&\pa_y\varphi^{+}(y,k,c)\end{matrix}\right),\ c_i>0,
\end{align}
and
\begin{equation}\label{eq-Wron-0}
  \begin{aligned}   
  &W(c_r,k)=\lim_{c_i\to0+}W(c,k)\\
  =&-\mathcal H\Big((b^{-1})''\Big)(c_r)+\int \frac{1}{(b(y)-c_r)^2}\left(\frac{1}{\phi_{1}^2(k,y,y_c)}-1\right) dy+i\pi(b^{-1})''(c_r),
  \end{aligned}
\end{equation}
where ${\mathcal H}$ is the Hilbert transform. By Lemma 5.4 of \cite{LiMasmoudiZhao2022critical}, $c$ is an eigenvalue of $\mathcal R_{b,k}$ if and only if $W(c,k)=0$.  

Recall that $b$ is odd and belongs to $\mathcal K^+$ class. By Theorem 1.5 of \cite{lin2003instability}, for any $0<k<k_*(T)$, there exists a purely imaginary eigenvalue $c=ic_i$ with $c_i>0$. Therefore, we  first focus on the zero point of $W(c,k)$ with $c_r=0$.

Since $\frac{b''}{b}$ is even, it follows that $\phi_{1}(y,k,0)$ is even as well. For $c_r=0$, one can easily check that if $\phi_2(y,k,c)$ satisfies \eqref{eq-phi2}, then $\bar\phi_2(-y,k,c)$ also satisfies \eqref{eq-phi2} with the same boundary condition. We write $\phi_2(y,k,c)=\phi_{2,r}(y,k,c)+i\phi_{2,i}(y,k,c)$. It holds that $\phi_{2,r}(y,k,c)$ is even $\phi_{2,i}(y,k,c)$ is odd. As a result, we have
\begin{align}\label{eq-psipm2}
  \phi(y,k,c)=-\bar\phi(-y,k,c),\quad \varphi^{+}(y,k,c)=\bar\varphi^{-}(-y,k,c), \text{ for }c_r=0.
\end{align}

Next, we show that $W_i(c,k)=0$ when $c_r=0$. We write
\begin{align*}
  W(c,k)=&\int_{-\infty}^{+\infty} \frac{1}{\phi^2(y',k,c)} dy'\\
  =&\int_{-\infty}^{+\infty} \frac{1}{\left(b(y')-c\right)^2}  dy'+\int_{-\infty}^{+\infty} \frac{1}{\left(b(y')-c\right)^2} \left(\frac{1}{\phi_1^2(y',k,0)\phi_2^2(y',k,c)}-1\right) dy'\eqdef I+II,
\end{align*}
both $I$ and $II$ are purely real. Indeed,
\begin{align*}
  I=&I_r+iI_i=\int_{-\infty}^{+\infty} \frac{1}{\left(b(y')-c\right)^2}  dy'  =\int_{-\infty}^{+\infty} \frac{1}{\left(v-c\right)^2}\pa_v \left(b^{-1}(v)\right)    dv= \int_{-\infty}^{+\infty} \frac{\pa_{v}^2 \left(b^{-1}(v)\right)}{v-c}   dv\\
  =&\int_{-\infty}^{+\infty} \frac{v\pa_{v}^2 \left(b^{-1}(v)\right)}{v^2+c_i^2}    dv +i\int_{-\infty}^{+\infty} \frac{c_i\pa_{v}^2 \left(b^{-1}(v)\right)}{v^2+c_i^2}   dv.
\end{align*}
As $b$ is odd, it is clear that $\pa_{v}^2 \left(b^{-1}(v)\right)$ is also odd, and then $I_i=0$. For  $II$,
\begin{align*}
  II=&II_r+iII_i=\int_{-\infty}^{+\infty} \frac{1}{\left(b(y')-c\right)^2} \left(\frac{1}{\phi_1^2(y',k,0)\phi_2^2(y',k,c)}-1\right) dy'\\
  =&\int_{-\infty}^{+\infty} \frac{b^2(y')-c_i^2+2ibc_i}{\left|b^2(y')+c_i^2\right|^2} \left(\frac{\phi_{2,r}^2(y',k,c)-\phi_{2,i}^2(y',k,c)-2i\phi_{2,r}(y',k,c)\phi_{2,i}(y',k,c)}{\phi_1^2(y',k,0)\left(\phi_{2,r}^2(y',k,c)+\phi_{2,i}^2(y',k,c)\right)^2}-1\right) dy',
\end{align*}
we have
\begin{align*}
  II_i=&\int_{-\infty}^{+\infty} \frac{2bc_i}{\left|b^2(y')+c_i^2\right|^2} \left(\frac{\phi_{2,r}^2(y',k,c)-\phi_{2,i}^2(y',k,c)}{\phi_1^2(y',k,0)\left(\phi_{2,r}^2(y',k,c)+\phi_{2,i}^2(y',k,c)\right)^2}-1\right) dy' \\
  &-\int_{-\infty}^{+\infty} \frac{b^2(y')-c_i^2}{\left|b^2(y')+c_i^2\right|^2} \left(\frac{2\phi_{2,r}(y',k,c)\phi_{2,i}(y',k,c)}{\phi_1^2(y',k,0)\left(\phi_{2,r}^2(y',k,c)+\phi_{2,i}^2(y',k,c)\right)^2}\right) dy'.
\end{align*}
Recall that $\phi_{2,r}(y,k,c)$ is even and $\phi_{2,i}(y,k,c)$ is odd. Consequently we have $II_i=0$, leading to $W_i(c)=0$ for $c=ic_i$. Therefore, to study the zero point of $W(c)$ with $c_r=0$, it suffices to to study the zero point of $W_r(ic_i)$. 

Next, to prove \eqref{eq-curve}, we show that 
\begin{align*}
  \pa_k W_r(ic_i,k)\approx-1,\quad \pa_{c_i}W_r(ic_i,k)\approx-\frac{1}{\gamma_0}, \text{ for } 0\le c_i\le\gamma_0\gamma_1\gamma_2,\ \frac{1}{2}\le k\le \frac{3}{2}.
\end{align*}
Here we will use some techniques from \cite{sinambela2023transition}.
  
Recall that $I_r$ is independent on $k$. We write
\begin{equation}\label{eq-Irci}
  \begin{aligned}   
  \pa_{c_i}I_r=&\int_{-\infty}^{+\infty} \pa_{c_i}\left(\frac{v}{v^2+c_i^2}\right)   \pa_{v}^2 \left(b^{-1}(v)\right) dv\\
  =&\int_{-\infty}^{+\infty} \pa_{v}\left(\frac{c_i}{v^2+c_i^2}\right)   \pa_{v}^2 \left(b^{-1}(v)\right) dv=-\int_{-\infty}^{+\infty}  \frac{c_i}{v^2+c_i^2}  \pa_{v}^3 \left(b^{-1}(v)\right) dv\\
  =&-\pi\pa_v^3b^{-1}(0)-\int_{-\infty}^{+\infty}  \frac{c_i}{v^2+c_i^2}  \left(\pa_{v}^3 \left(b^{-1}(v)\right)-\pa_v^3b^{-1}(0)\right) dv.    
  \end{aligned}
\end{equation}
For the second term, by using Hardy's inequality we have
\begin{align*}
  &\left|\int_{-\infty}^{+\infty}  \frac{c_i}{v^2+c_i^2}  \left(\pa_{v}^3 \left(b^{-1}(v)\right)-\pa_v^3b^{-1}(0)\right) dv\right|\\
  \le&\int_{|v|\le N c_i} \left|\frac{c_i}{v^2+c_i^2} \left(\pa_{v}^3 \left(b^{-1}(v)\right)-\pa_v^3b^{-1}(0)\right)\right| dv\\
  &+\int_{|v|> N c_i} \left|\frac{c_i}{v^2+c_i^2}  \pa_{v}^3 \left(b^{-1}(v)\right)\right|+\left|\frac{c_i}{v^2+c_i^2}\pa_v^3b^{-1}(0) \right| dv\\
  \le& \left(\int_{|v|\le N c_i} \left|\frac{vc_i}{v^2+c_i^2}\right|^2 dv \right)^{\frac{1}{2}}\left(\int_{|v|\le N c_i} \left|\pa_{v}^4 \left(b^{-1}(v)\right)\right|^2 dv \right)^{\frac{1}{2}}\\
  &+\int_{|v|> N c_i} \left|\frac{c_i}{v^2+c_i^2}  \pa_{v}^3 \left(b^{-1}(v)\right)\right|+\left|\frac{c_i}{v^2+c_i^2}\pa_v^3b^{-1}(0) \right| dv.
\end{align*}
It is easy to check that
\begin{align*}
   \pa_v^3b^{-1}(v)=\frac{2\left(b''(y)\right)^2-b'''(y)b'(y)}{\left(b'\left(y\right)\right)^4}|_{y=b^{-1}(v)}\\
  \pa_v^4b^{-1}(v)=\frac{-b''''(y)\left(b'(y)\right)^2+8b'''(y)b''(y)b'(y)-10 \left(b''(y)\right)^3}{\left(b'\left(y\right)\right)^6}|_{y=b^{-1}(v)}
\end{align*}
and
\begin{align*}
  \left|\pa_v^3b^{-1}(v)\right|\approx  \frac{1}{\gamma_0}e^{-C\frac{v^2}{\gamma_0^2}}+\frac{v^2}{\gamma_0^3}e^{-C\frac{v^2}{\gamma_0^2}} \lesssim \frac{1}{\gamma_0}\\
  \left|\pa_v^4b^{-1}(v)\right|\approx \frac{v}{\gamma_0^3}e^{-C\frac{v^2}{\gamma_0^2}}+\frac{v^3}{\gamma_0^5}e^{-C\frac{v^2}{\gamma_0^2}}+\gamma_2\frac{v}{\gamma_0^3\gamma_1^2}e^{-C\frac{v^2}{\gamma_0^2\gamma_1^2}}+\gamma_2\frac{v^3}{\gamma_0^5\gamma_1^4}e^{-C\frac{v^2}{\gamma_0^2\gamma_1^2}}. 
\end{align*}

It is follows that 
\begin{align*}
  \int_{|v|> N c_i} \left|\frac{c_i}{v^2+c_i^2}  \pa_{v}^3 \left(b^{-1}(v)\right)\right|+\left|\frac{c_i}{v^2+c_i^2}\pa_v^3b^{-1}(0) \right| dv\lesssim \frac{1}{N}\frac{1}{\gamma_0}.
\end{align*}
and
\begin{align*}
  \left(\int_{|v|\le N c_i} \left|\frac{vc_i}{v^2+c_i^2}\right|^2 dv \right)^{\frac{1}{2}}\lesssim |c_i|^{\frac{1}{2}},
\end{align*}
and
\begin{align*}
  \left(\int_{|v|\le N c_i} \left|\pa_{v}^4 \left(b^{-1}(v)\right)\right|^2 dv \right)^{\frac{1}{2}}\lesssim \gamma_2\max(\frac{N^{\frac{3}{2}}|c_i|^{\frac{3}{2}}}{\gamma_0^3\gamma_1^2},\frac{N^{\frac{7}{2}}|c_i|^{\frac{7}{2}}}{\gamma_0^5\gamma_1^4}),
\end{align*}
Therefore,
\begin{align*}
  \left(\int_{|v|\le N c_i} \left|\frac{vc_i}{v^2+c_i^2}\right|^2 dv \right)^{\frac{1}{2}}\left(\int_{|v|\le N c_i} \left|\pa_{v}^4 \left(b^{-1}(v)\right)\right|^2 dv \right)^{\frac{1}{2}}\lesssim \gamma_2\max(\frac{N^{\frac{3}{2}}|c_i|^{2}}{\gamma_0^3\gamma_1^2},\frac{N^{\frac{7}{2}}|c_i|^{4}}{\gamma_0^5\gamma_1^4}).
\end{align*}

By taking $1\ll N\ll \frac{1}{\gamma_2}$, we have
\begin{align*}
  \left|\int_{-\infty}^{+\infty}  \frac{c_i}{v^2+c_i^2}  \left(\pa_{v}^3 \left(b^{-1}(v)\right)-\pa_v^3b^{-1}(0)\right) dv\right|\ll\frac{1}{\gamma_0}.
\end{align*}
Then, it follows from \eqref{eq-Irci} and the fact $\pa_v^3b^{-1}(0)\approx -\frac{1}{\gamma_0}$ that
\begin{align*}
  \pa_{c_i}I_r\approx -\frac{1}{\gamma_0}.
\end{align*}

Next, we will give the estimate for $\pa_{c_i}II_r$ and show that $|\pa_{c_i}II|\lesssim \frac{1}{N}\frac{1}{\gamma_0}$, which gives $\pa_{c_i}W_r(c,k)\approx -\frac{1}{\gamma_0}$.

We introduce the following good derivative
\begin{align*}
  \pa_G=\pa_{c_i}+i \frac{\pa_y}{b'(0)}.
\end{align*}
We then deduce
\begin{align*}
  \pa_{c_i}II=&\int_{-\infty}^{+\infty} \pa_G\left(\frac{1}{\left(b(y')-c\right)^2} \left(\frac{1}{\phi_1^2(y',k,0)\phi_2^2(y',k,c)}-1\right) \right)dy'\\
  =&\int_{-\infty}^{+\infty}\left(\frac{1}{\phi_1^2(y',k,0)\phi_2^2(y',k,c)}-1\right) \pa_G\left(\frac{1}{\left(b(y')-c\right)^2} \right)\\
  +& \frac{1}{\left(b(y')-c\right)^2}\pa_G \left(\frac{1}{\phi_1^2(y',k,0)\phi_2^2(y',k,c)}-1\right)  dy'\\
  \eqdef&II_a+II_b.
\end{align*}
We have
\begin{align*}
  \left|\pa_G\left(\frac{1}{\left(b(y')-c\right)^2} \right)\right|= \left|2i \frac{1-\frac{b'(y')}{b'(0)}}{\left(b(y')-c\right)^3}\right|\lesssim \frac{\left\|b''\right\|_{L^\infty}}{{y'}^2+c_i^2}\lesssim\frac{1}{{y'}^2+c_i^2}.
\end{align*}
Thus by Proposition \ref{prop-phi}, we obtain
\begin{align*}
  \left|II_a\right|\lesssim 1.
\end{align*}

Let $\phi_*(y,k,c)=\phi_1(y,k,c)\phi_2(y,k,c)$ which satisfies
\begin{equation}
  \left\{
    \begin{array}{ll}
      \left(\phi_*'(y,k,c) \left(b(y)-ic_i\right)^2\right)'-k^2\phi_*(y,k,c)\left(b(y)-ic_i\right)^2=0,\\
      \phi_*(0,k,c)=1,\qquad\phi_*'(0,k,c)=0.
    \end{array}
  \right.
\end{equation}
Apply $\pa_G$ to both sides of the above equation, we get
\begin{equation}\label{eq-phi-G}
    \begin{aligned}   
  \left( \left(b(y)-ic_i\right)^2 \left(\pa_G\phi_*\right)'(y,k,c)\right)'=& k^2\left(\pa_G\phi_*(y,k,c)\right)\left(b(y)-ic_i\right)^2\\
  &- \left( \frac{\pa_G\left(b(y)-ic_i\right)^2}{\left(b(y)-ic_i\right)^2}\right)'\left(b(y)-ic_i\right)^2\phi_*'(y,k,c),     
    \end{aligned}
  \end{equation}  
with boundary condition
\begin{align*}
  \pa_G\phi_*(0,k,c)=0,\qquad(\pa_G\phi_*)'(0,k,c)=\frac{k^2}{b'(0)}.
\end{align*}

We have
\begin{align*}
  &\left( \left(b(y)-ic_i\right)^2 \left(\frac{\pa_G\phi_*(y,k,c)}{\phi_*(y,k,c)}\right)'\phi_*^2(y,k,c)\right)'\\
  =&- \left( \frac{\pa_G\left(b(y)-ic_i\right)^2}{\left(b(y)-ic_i\right)^2}\right)'\left(b(y)-ic_i\right)^2\phi_*'(y,k,c)\phi_*(y,k,c),
\end{align*}
and then
\begin{align*}
  &\frac{\pa_G\phi_*(y,k,c)}{\phi_*(y,k,c)}\\
  =&-\int_{0}^y \frac{c_i^2\frac{k^2}{b'(0)}}{\left(b(y')-ic_i\right)^2\phi_*^2(y',k,c)}  d y'\\
  &-\int_{0}^y \frac{\int_{0}^{y'} \left( \frac{\pa_G\left(b(y'')-ic_i\right)^2}{\left(b(y'')-ic_i\right)^2}\right)'\left(b(y'')-ic_i\right)^2\phi_*'(y'',k,c)\phi_*(y'',k,c) d y''}{\left(b(y')-ic_i\right)^2\phi_*^2(y',k,c)} d y'\\
  \eqdef&G_1+G_2.
\end{align*}
For $G_1$, it follows from Proposition \ref{prop-phi} that
\begin{align*}
  \left|G_1\right|\le \int_{0}^y \frac{c_i^2\frac{k^2}{b'(0)}}{\left|b(y')-ic_i\right|^2 \left|\phi_*(y',k,c)\right|^2}  d y'\lesssim \int_{0}^y \frac{c_i^2 }{(y')^2+c_i^2}d y' \lesssim \left|c_i\right|.
\end{align*}
For $G_2$, as $|b''|$ has uniform bound, we have
\begin{align*}
  \left|\left( \frac{\pa_G\left(b(y'')-ic_i\right)^2}{\left(b(y'')-ic_i\right)^2}\right)'\right|=&\left|\left( \frac{2i \left(\frac{b'(y'')}{b'(0)}-1\right)}{ b(y'')-ic_i }\right)'\right|=\left|2i\frac{\frac{b''(y'')}{b'(0)}\left(b(y'')-ic_i\right)-b'(y'')\left(\frac{b'(y'')}{b'(0)}-1\right)}{\left(b(y'')-ic_i\right)^2}\right|\\
  \lesssim& \frac{1}{\left|b(y'')-ic_i\right|}.
\end{align*}
Recall that $b$ is monotonic. By Proposition \ref{prop-phi}, we have
\begin{align*}
  \left|G_2\right|\lesssim& \left|\int_{0}^y \int_{0}^{y'} \left( \frac{\pa_G\left(b(y'')-ic_i\right)^2}{\left(b(y'')-ic_i\right)^2}\right)' \frac{\left(b(y'')-ic_i\right)^2\phi_*'(y'',k,c)\phi_*(y'',k,c)}{\left(b(y')-ic_i\right)^2\phi_*^2(y',k,c)}  d y'' d y'\right|\\
  \lesssim&\int_{0}^y \int_{0}^{y'} \frac{ |y''|+|c_i| }{\left|b(y')-ic_i\right|}  d y'' d y'\lesssim \left|y\right|^2 .
\end{align*}

As a result, it holds that 
\begin{align}\label{eq-phi*ci}
  \left|\frac{\pa_G\phi_*(y,k,c)}{\phi_*(y,k,c)}\right|\lesssim \left|c_i\right| + \left|y\right|^2 .
\end{align}

Then, we have
\begin{align*}
  \left|II_b\right|=&\left|\int_{-\infty}^{+\infty} \frac{1}{\left(b(y')-ic_i\right)^2} \frac{\pa_G\phi_*(y',k,c)}{\phi_*(y',k,c)} \frac{1}{\phi_*^2(y',k,c)} dy'\right|\\
  \lesssim&\int_{-\infty}^{+\infty}\left| \frac{\left|c_i\right| + \left|y\right|^2}{\left|b(y')-ic_i\right|^2}  \frac{1}{\left|\phi_*^2(y',k,c)\right|} \right|dy'\lesssim 1.
\end{align*}

Combing the estimate for $\pa_{c_i}I_r$ and $\left|\pa_{c_i}II\right|$, we have
\begin{align}\label{eq-W-ci}
  \pa_{c_i}W_r(c,k)\approx\pa_{c_i}I_r\approx-\frac{1}{\gamma_0}.
\end{align}

Next, we give the estimate for $\pa_k W_r(c,k)$. It is clear that
\begin{align*}
  \pa_k W_r(c,k)=\mathrm{Re}(\pa_kII).
\end{align*}
We have
\begin{align*}
  \pa_{k}II =&\int_{-\infty}^{+\infty} \frac{1}{\left(b(y')-c\right)^2}\pa_k \left(\frac{1}{\phi_1^2(y',k,0)\phi_2^2(y',k,c)}-1\right)  dy'\\
  =&-2\int_{-\infty}^{+\infty} \frac{1}{\left(b(y')-c\right)^2} \frac{\pa_k\phi_1(y',k,0)}{\phi_1(y',k,0)}\frac{1}{\phi_1^2(y',k,0)\phi_2^2(y',k,c)} dy'\\
  &-2\int_{-\infty}^{+\infty} \frac{1}{\left(b(y')-c\right)^2} \frac{\pa_k\phi_2(y',k,0)}{\phi_2(y',k,0)}\frac{1}{\phi_1^2(y',k,0)\phi_2^2(y',k,c)} dy'\\
  \eqdef&K_1+K_2.
\end{align*}

We first study $K_1$. Similar to \eqref{eq-phi-G}, we derive the equation for $\pa_k\phi_1(y,k,c)$,
\begin{align*}
  \left( b^2(y)\left(\frac{\pa_k\phi_1(y,k,0)}{\phi_1(y,k,0)}\right)'\phi_1^2(y,k,0)\right)'=&2k \phi_1^2(y,k,0) b^2(y) ,
\end{align*}
with boundary condition
\begin{align*}
  \pa_k\phi_1(0,k,0)=0,\qquad(\pa_k\phi_1)'(0,k,0)=0.
\end{align*}
It follows that
\begin{align*}
  &\frac{\pa_k\phi_1(y,k,0)}{\phi_1(y,k,0)}=2\int_{0}^y \frac{1}{b^2(y')\phi_1^2(y',k,0)}\int_{0}^{y'}  2kb^2(y'')\phi_1^2(y'',k,0) d y'' d y'.
\end{align*}
Thus, it is clear that, $\pa_k\phi_1(y,k,0)\ge0$. Moreover, by Proposition \ref{prop-phi}, we have
\begin{align}\label{eq-est-phi1-k}
  \frac{\pa_k\phi_1(y,k,0)}{\phi_1(y,k,0)}\ge \frac{1}{C}|y|^2,\text{ for }|y|\le1,\qquad\frac{\pa_k\phi_1(y,k,0)}{\phi_1(y,k,0)}\le C|y|^2,\text{ for }y\in\mathbb R,\\
  \left| \left(\frac{\pa_k\phi_1(y,k,0)}{\phi_1(y,k,0)}\right)'\right|\le C|y|,\text{ for }y\in\mathbb R.\label{eq-est-phi1-k2}
\end{align}

Then we write
\begin{align*}
  K_1=&-2\int_{-\infty}^{+\infty} \frac{b^2(y')-c_i^2+2ib(y')c_i}{\left(b^2(y')+c_i^2\right)^2} \frac{\pa_k\phi_1(y',k,0)}{\phi_1(y',k,0)}\\
  &\qquad\qquad\qquad\cdot\frac{\phi_{2,r}^2(y',k,c)-\phi_{2,i}^2(y',k,c)-2i\phi_{2,r}(y',k,c)\phi_{2,i}(y',k,c)}{\phi_1^2(y',k,0)\phi_2^2(y',k,c)} dy'.
\end{align*}
It holds that
\begin{align*}
  \mathrm{Re}(K_1)=&-2\int_{-\infty}^{+\infty} \frac{b^2(y')}{\left(b^2(y')+c_i^2\right)^2} \frac{\pa_k\phi_1(y',k,0)}{\phi_1(y',k,0)}\frac{\phi_{2,r}^2(y',k,c)-\phi_{2,i}^2(y',k,c)}{\phi_1^2(y',k,0)\phi_2^2(y',k,c)} dy'\\
  &+2\int_{-\infty}^{+\infty} \frac{c_i^2}{\left(b^2(y')+c_i^2\right)^2} \frac{\pa_k\phi_1(y',k,0)}{\phi_1(y',k,0)}\frac{\phi_{2,r}^2(y',k,c)-\phi_{2,i}^2(y',k,c)}{\phi_1^2(y',k,0)\phi_2^2(y',k,c)} dy'\\
  &-8\int_{-\infty}^{+\infty} \frac{b(y')c_i}{\left(b^2(y')+c_i^2\right)^2} \frac{\pa_k\phi_1(y',k,0)}{\phi_1(y',k,0)}\frac{\phi_{2,r}(y',k,c)\phi_{2,i}(y',k,c)}{\phi_1^2(y',k,0)\phi_2^2(y',k,c)} dy'\\
  \eqdef&K_{1,1}+K_{1,2}+K_{1,3}.
\end{align*}
Recall Proposition \ref{prop-phi}. One can easily check that
\begin{align*}
  K_{1,1}\approx-1,\qquad \left|K_{1,2}\right|+\left|K_{1,3}\right|\lesssim c_i.
\end{align*}

Next, we turn to $K_2$. Also similar to \eqref{eq-phi-G}, we derive the equation for $\pa_k\phi_2(y,k,c)$ (we brief $\phi_1(y,k,c)$, $\phi_2(y,k,c)$ to $\phi_1(y)$, $\phi_2(y)$), 
\begin{align*}
  &\left( \left(b(y)-ic_i\right)^2 \phi_1^2(y)\left(\frac{\pa_k\phi_2(y)}{\phi_2(y)}\right)'\phi_2^2(y)\right)'\\
  =&\frac{2ic_i b'(y) \left(b(y)- ic_i\right)}{b(y)}\pa_k\phi_{1}(y)\phi_{1}'(y)\phi_{2}^2(y)-\frac{2ic_i b'(y) \left(b(y)- ic_i\right)}{b(y)}\phi_{1}(y)\pa_k\phi_{1}'(y)\phi_{2}^2(y)\\
  &-2\frac{\pa_k\phi_1'(y)\phi_1(y)-\phi_1'(y)\pa_k\phi_1(y)}{\phi_1^2(y)}\left(b(y)-ic_i\right)^2\phi_2'(y)\phi_2(y)\phi_1^2(y),
\end{align*}
with boundary condition
\begin{align*}
  \pa_k\phi_2(0)=0,\qquad(\pa_k\phi_2)'(0)=0.
\end{align*}
It holds that
\begin{align*}
  \frac{\pa_k\phi_2(y)}{\phi_2(y)}=&\int_{0}^y \frac{\int_{0}^{y'} \frac{2ic_i b'(y'') \left(b(y'')- ic_i\right)}{b(y'')}\pa_k\phi_{1}(y'')\phi_{1}'(y'')\phi_{2}^2(y'')  d y''}{\left(b(y')-ic_i\right)^2\phi_1^2(y')\phi_2^2(y')} d y'\\
  &-\int_{0}^y \frac{\int_{0}^{y'} \frac{2ic_i b'(y'') \left(b(y'')- ic_i\right)}{b(y'')}\phi_{1}(y'')\pa_k\phi_{1}'(y'')\phi_{2}^2(y'')  d y''}{\left(b(y')-ic_i\right)^2\phi_1^2(y')\phi_2^2(y')} d y'\\
  &-\int_{0}^y \frac{\int_{0}^{y'} 2\left(\pa_k\phi_1'(y'')\phi_1(y'')-\phi_1'(y'')\pa_k\phi_1(y'')\right)\left(b(y'')-ic_i\right)^2\phi_2'(y'')\phi_2(y'') d y''}{\left(b(y')-ic_i\right)^2\phi_1^2(y')\phi_2^2(y')} d y'\\
  =&-\int_{0}^y \frac{\int_{0}^{y'} \frac{2ic_i b'(y'') \left(b(y'')- ic_i\right)}{b(y'')}\phi_{1}^2(y'') \left(\frac{\pa_k\phi_{1}(y'')}{\phi_{1}(y'')}\right)'\phi_{2}^2(y'')  d y''}{\left(b(y')-ic_i\right)^2\phi_1^2(y')\phi_2^2(y')} d y'\\
  &-\int_{0}^y \frac{\int_{0}^{y'} 2\left(\frac{\pa_k\phi_{1}(y'')}{\phi_{1}(y'')}\right)'\phi_1^2(y'')\left(b(y'')-ic_i\right)^2\phi_2'(y'')\phi_2(y'') d y''}{\left(b(y')-ic_i\right)^2\phi_1^2(y')\phi_2^2(y')} d y'.
\end{align*}
By using \eqref{eq-est-phi1-k2} and Proposition \ref{prop-phi}, one can check that
\begin{align*}
  \left|\frac{\pa_k\phi_2(y)}{\phi_2(y)}\right|\lesssim c_i(|y|+|y|^3).
\end{align*}
Then we can see that
\begin{align*}
  \left|K_2\right|\lesssim&\int_{-\infty}^{+\infty} \frac{c_i(|y|+|y|^2)}{\left|b(y')-c\right|^2} \frac{1}{\phi_1^2(y',k,0)} dy'\lesssim \left|c_i\ln c_i\right|.
\end{align*}

Combing the above estimates, we get
\begin{align*}
  \pa_k W_r(c,k)=\mathrm{Re}(\pa_kII)\approx-1.
\end{align*}

Then we finished the proof of \eqref{eq-curve}.

The proof for existence of unstable eigenvalue in \cite{lin2003instability} inherently includes the locally uniqueness of the implicit curve of Wronskian (see also \cite{sinambela2023transition}). Since the background flow we are considering is odd, the proof of the existence of the unstable eigenvalue is relatively simple. Below, we provide a brief explanation of the uniqueness of the implicit curve. 

The Wronskian $W(c,k)$ is not necessarily analytic with respect to $c$. However, from  Remark B.1 of \cite{LiMasmoudiZhao2022critical}, there exists $Y$ such that $f(c)= \frac{1}{\varphi^+(Y,k,c)\varphi^-(-Y,k,c)}$ has no zeros, and
\begin{align*}
  \mathcal W(c,k)=f(c)W(c,k)
\end{align*}
is analytic for $c_i>0$ and continuous for $c_i\ge0$. Moreover $\mathcal W(c,k)$ has same zeros as $W(c,k)$. Here $\varphi^\pm$ is given in \eqref{eq-psipm}.  By \eqref{eq-psipm2} we know that $f(c)$ is real at $c_r=0$. 

Using \eqref{eq-W-ci} and the Cauchy–Riemann equations, for $c=ic_i$, $0<c_i\le\gamma_0\gamma_1\gamma_2$, we have
\begin{equation}\label{eq-CR}
  \begin{aligned}   
  \pa_{c_r}\mathcal W_i(c,k)=&-\pa_{c_i}\mathcal W_r(c,k)=-\pa_{c_i}f(c)W_r(c,k)-f(c)\pa_{c_i}W_r(c,k)\\
  =&-f(c)\pa_{c_i}W_r(c,k)\neq0,\\
  \pa_{c_r}\mathcal W_r(c,k)=&\pa_{c_i}\mathcal W_i(c,k)=\pa_{c_i} \left(f(c)W_r(c,k)\right)=0.
  \end{aligned}
\end{equation}
From \eqref{eq-phi*ci} and Proposition \ref{prop-phi}, one can easily check that $\left|\pa_{c_i}f(c)\right|$ has uniform bound for $c=ic_i$, $0\le c_i\le\gamma_0\gamma_1\gamma_2$ (refer to Remark B.1 of \cite{LiMasmoudiZhao2022critical}). Consequently, \eqref{eq-CR}  holds for $c=0$ as well.  By employing the two-dimensional implicit function theorem, we conclude that for each $k$, there exists only one zero point $c(k)$ such that $\mathcal W_r(c(k),k)=0$, and the implicit cure is given in \eqref{eq-curve}.

This complete the proof of Proposition \ref{lem-quanti}. By taking $\gamma=\frac{1}{C}\gamma_0\gamma_1\gamma_2$, the result of Theorem \ref{thm} follows from Proposition \ref{lem-neutral} and \ref{lem-quanti}.

\begin{appendix}
\section{Estimate of solution to homogeneous Rayleigh}
In this appendix, we introduce a special solution of \eqref{eq-Rayleigh} and list its properties, which are widely used throughout the proofs in this paper.
  \begin{proposition}\label{prop-phi}
  For monotonic shear flow $b(y)$ and wave number $k\neq0$, there exist $0<c_{0,k}\le1$, such that for $c=c_r+ic_i\in\mathbb C$ with $0\le c_i\le c_{0,k}$, equation \ref{eq-Rayleigh} has a regular solution
\begin{align*}
  \phi(y,k,c)=\left(b(y)-c\right)\phi_1(y,k,c_r)\phi_2(y,k,c),
\end{align*}
  which satisfies $\phi(y,k,c)|_{y=y_c}=ic_i$ and $\phi'(y,k,c)|_{y=y_c}=b'(y_c)$ with $y_c=b^{-1}(c_r)$. Here $\phi_1(y,k,c_r)$ is a real function that solves
  \begin{equation}\label{eq-phi1}
    \left\{
      \begin{array}{ll}
        \pa_y\left(\left(b(y)-c_r\right)^2\phi'_{1}(y,k,c_r)\right)-k^2\left(b(y)-c_r\right)^2 \phi_{1}(y,k,c_r)=0,\\
       \phi_{1}(y,k,c_r)|_{y=y_c}=1,\quad\pa_y \phi_{1}(y,k,c_r)|_{y=y_c}=0,
      \end{array}
    \right.
  \end{equation}
and $\phi_2(y,k,c)$ solves
\begin{equation}\label{eq-phi2}
  \left\{
    \begin{array}{l}
      \pa_y\Big(\left(b(y)-c\right)^2\phi_{1}^2(y,k,c_r)\phi_{2}'(y,k,c)\Big) +\frac{2ic_i b'(y) \left(b(y)-c\right)}{b(y)-c_r}\phi_{1}(y,k,c_r)\phi_{1}'(y,k,0)\phi_{2}(y,k,c)=0,\\
      \phi_{2}(y,k,c)|_{y=y_c}=1,\quad\phi'_{2}(y,k,c)|_{y=y_c}=0.
    \end{array}
  \right.
\end{equation}
It holds that
\begin{align}
      &\phi_{1}(y,k,c_r)\ge1,\quad (y-y_c)\phi'_{1}(y,k,c_r)\geq 0,\label{eq-phi1-est-1}\\
   &C^{-1}e^{C^{-1}|k|(|y-y_c|)}\le\phi_{1}(y,k,c_r)\le Ce^{C|k||y-y_c|},\label{eq-phi1-est-2}\\
    &|\phi'_{1}(y,k,c_r)|\leq C|k|\min\{|k|(y-y_c),1\}\phi_{1}(y,k,c_r),\label{eq-phi1-est-3}\\
    &|\phi''_{1}(y,k,c_r)|\leq Ck^2\phi_{1}(y,k,c_r),\\
    & 0\leq \phi_{1}(y,k,c_r)-1\le C\min \{1, |k|^2|y-y_c|^2\}\phi_{1}(y,k,c_r),\label{eq-phi1-est-4}\\
    &\frac{\phi_{1}(y,k,c_r)}{\phi_{1}(y_1,k,c_r)}\le Ce^{-C|k|(y-y_1)}\text{ for }y_1\le y\le y_c,\label{eq-phi1-est-5}\\
     &\frac{\phi_{1}(y,k,c_r)}{\phi_{1}(y_1,k,c_r)}\le Ce^{-C|k|(y_1-y)}\text{ for }y_1\ge y\ge y_c,  \label{eq-phi1-est-6}
\end{align}
  \begin{align}
    &|\phi_2(y,k,c)-1|\le C\min \{|k|c_i,|k|^2c_i|y-y_c|, |k|^2|y-y_c|^2\},\label{eq-phi2-est-1}\\
      &|\phi'_2(y,k,c)|\le C|k|^2\min \{c_i, |y-y_c|\},\label{eq-phi2-est-2}\\
      & \|\phi''_2(y,k,c)\|_{L^\infty_y}\le C|k|^2,\label{eq-phi2-est-3}\\
      & C^{-1}|y-c|e^{C|k||y-y_c|}\le|\phi(y,k,c)|\le |y-c|e^{C|k||y-y_c|},\label{eq-phi-est-0}
  \end{align} 
  where $C>0$ is a constant depends only on $\sup_y b'(t,y)$ and $\inf_y b'(t,y)$.
\end{proposition}
For the proof, we refer the readers to Proposition 5.3 of \cite{LiMasmoudiZhao2022critical} (also, Proposition 2.1 of \cite{li2023asymptotic}).
\end{appendix}

\bibliographystyle{siam.bst} 
\bibliography{references.bib}

\end{document}